\documentclass[12pt,a4paper]{article}
\usepackage{natbib}
\usepackage{float}
\usepackage{multirow,bbm}
\usepackage{amssymb,amsmath, amsfonts, booktabs}
\usepackage{lscape}
\usepackage{enumerate}
\usepackage{amsthm}
\usepackage[T1]{fontenc}
\usepackage[utf8]{inputenc}
\usepackage{subcaption}
\usepackage{lmodern}
\usepackage{placeins}
\usepackage[english]{babel}
\usepackage{pdflscape}
\usepackage{csquotes}
\usepackage[colorlinks, citecolor=blue]{hyperref}
\usepackage{mathtools}
\usepackage{xcolor}
\numberwithin{equation}{section}
\usepackage{graphicx}
\usepackage{caption}
\usepackage{subcaption}
\usepackage{upgreek}
\bibpunct[, ]{(}{)}{;}{a}{,}{,}
\newtheorem{theorem}{Theorem}[section]
\newtheorem{exe}{Example}[section]

\setcitestyle{authoryear,open={(},close={)}} 
\begin{document}
	
	\begin{center}
		{\Large\bf On the Study of (Dynamic) Weighted Fractional Generalized Cumulative Residual Inaccuracy  with Applications} \\
		\vspace{0.3in}
		
		{\bf Aman Pandey~~~~ \bf Chanchal Kundu\footnote{Corresponding author e-mail: ckundu@rgipt.ac.in, chanchal\_kundu@yahoo.com.}}
		\\
		Department of Mathematical Sciences\\ Rajiv Gandhi Institute of Petroleum Technology\\Jais  229304, U.P., India\vspace{0.1in} \\{February, 2026}\\
	\end{center}
	\begin{abstract}
Recently, there has been growing interest in developing information-theoretic measures that capture uncertainty and inaccuracy. In this work, we introduce a novel measure, referred to as the weighted fractional generalized cumulative residual inaccuracy (WFGCRI), as an extension of the weighted cumulative residual inaccuracy (WCRI) to capture complex tail behavior and structural discrepancies between the true and assumed distributions. We derive its fundamental analytical properties and establish several bounds that highlight its key relationships. Further, we consider a dynamic version of WFGCRI, and its temporal evolution is examined under the proportional hazard rate and proportional odds models. An empirical estimation procedure for WFGCRI is developed, and its finite-sample performance is systematically assessed through Monte Carlo simulation studies. Moreover, the ability of WFGCRI to detect chaotic dynamics is demonstrated using the Ricker and Tent maps. The practical relevance of WFGCRI is further established through an analysis of a real dataset for uncertainty assessment. Additionally, we demonstrate that WFGCRI is superior to the traditional WCRI when applied to stock market data. 
	\end{abstract}
	\vskip 2mm
	
	\noindent \textbf{Keywords:} Fractional generalized cumulative residual entropy; Inaccuracy;  Non-parametric\\ estimation; Proportional hazard rate model;  Stochastic orders.\\
	\noindent \textbf{Mathematics Subject Classification (2020):} Primary: 62B10; Secondary: 60E15. \\
	\noindent------------------------------------------------------------------------------------------------------------------------------
	\pagestyle{myheadings}
	
	\thispagestyle{empty}
	\section{Introduction}
Shannon entropy, introduced by \cite{Sh1948}, quantifies the uncertainty in a random variable (RV) and serves as a foundational concept in information theory. Let $X$ be a non-negative absolutely continuous RV with probability density function (pdf)
$f_X$. The  differential entropy of $X$ is defined as:
\begin{align}
H(X) = -\int_0^\infty f_X(w)\ln f_X(w)\,dw.
\end{align}
The entropy measures play a central role in data compression, communication theory, and statistical learning \citep{Mac2003}. Beyond engineering, Shannon entropy has found applications across various fields, from physics to neuroscience, underscoring its versatility in quantifying uncertainty and information. 

\cite{Rao2004} introduced an alternative measure of entropy called cumulative residual entropy (CRE) which is defined for the RV $X$ as:
 \begin{align}\label{eq:1.2}
\mathcal {H}(X)=-\int_0^{\infty}S_X (w)\ln {S_X}(w)dw,
 \end{align}
where $S_X=1-F_X$ is the survival function ({sf}) with $F_X$ as the cumulative distribution function (cdf) of $X$. Further, they studied the CRE and  established several theoretical properties, highlighting its advantages over Shannon entropy. The paper also demonstrated CRE's potential through applications in reliability engineering and computer vision.

While Shannon entropy is fundamentally additive in nature, such additivity may be restrictive in scenarios characterized by nonextensivity, long-range interactions, or multi-fractal structures. To overcome these limitations, various generalized entropy frameworks have been proposed, most notably the one-parameter families of \cite{Tsa1988} and \cite{Ren1961} entropies. Within this line of inquiry, \cite{Ubri2009} introduced an alternative information measure, termed as fractional differential entropy, given by
\begin{align}\label{eq:1.3}
 H_\beta(X) = \int_0^\infty f_X(w)\left(-\ln f_X(w)\right)^\beta\,dw, \quad 0 \leq \beta \leq 1.
\end{align}
This formulation preserves essential properties such as positivity and concavity, yet it deviates from additivity in general settings. Importantly, the measure coincides with Shannon entropy for $\beta=1$. Fractional differential entropy incorporates a nonlinear transformation of probabilistic weights, thereby providing enhanced flexibility in representing complex forms of uncertainty. 

Following CRE \eqref{eq:1.2} and \eqref{eq:1.3}, \cite{Xio2019} proposed the notion of  fractional cumulative residual entropy (FCRE) for the  RV $X$, defined as:
\begin{align}\label{1.4}
 {\mathcal H}_\beta(X) = \int_0^{\infty} S_X(w) \left(-\ln S_X(w)\right)^\beta \, dw, \quad 0 \leq \beta \leq 1,
\end{align}
which generalizes the CRE by incorporating a fractional power of the information function.  The authors derived theoretical properties, bounds and, a consistent empirical estimator. 

\cite{Mach2014} further extended Ubriaco entropy \eqref{eq:1.3} and demonstrated that fractional entropy offers heightened sensitivity to signal dynamics, enabling richer characterization of system evolution compared to traditional single-parameter entropy measures. Building upon this idea, \cite{Di2021}  proposed an extension of FCRE \eqref{1.4} to achieve a better correspondence with other useful information measures by introducing an additional term. This led to the definition of the fractional generalized cumulative residual entropy (FGCRE), which provides a more flexible and comprehensive framework for quantifying uncertainty, given as:
\begin{align}\label{eq:5}
	\tilde{\mathcal H}_\beta(X)=\frac{1}{\Gamma(\beta+1)}\int_0^\infty S_X(w)\left(-\ln S_X(w)\right)^\beta \, dw, \quad  \beta \geq 0.
\end{align} 
They further extended the measure to its cumulative form by replacing the sf with the cdf.

The FGCRE measure (\ref{eq:5}) is invariant under location shifts. That is, for any constant $a$,
the FGCRE remains unchanged when RV $X$ is shifted by $a$. The property of invariance is the fact that the FGCRE depends solely on the sf $S_X$, and not on the particular values of the RV. While this property is mathematically pleasing, it
might be restrictive in many actual applications, e.g., reliability engineering or computational
neuroscience, where the actual magnitudes of the variables matter. To eliminate this limitation,  \cite{alm2023} introduced the weighted version of FGCRE (WFGCRE). The WFGCRE is expressed as:
\begin{align}\label{eq::1.6}
	 \tilde{\mathcal H}_\beta(X;{\psi})= \frac{1}{\Gamma(\beta+1)} \int_0^{\infty}\psi(w) S_X(w) \left(-\ln S_X(w)\right)^\beta \, dw, \quad \beta \geq0.
\end{align}
    
\cite{Ker1961} introduced the concept of an inaccuracy measure, referred to as the Kerridge inaccuracy, to quantify the discrepancy between the true distribution and a hypothesized distribution. This measure is defined as:
\begin{align}
	\mathcal{K}(X, Y) = -\int_0^\infty f_X(w) \ln f_Y(w)\,dw,
\end{align}
where $f_X$ and $f_Y$ represent the pdfs of  true (X) and assumed distributions (Y), respectively. It is especially relevant in statistical inference and reliability analysis when the true model is not precisely known. The inaccuracy measure helps quantify how closely a proposed model $f_Y$ aligns with the actual data-generating process $f_X$.
Metrics like Kerridge’s inaccuracy and Shannon entropy are essential tools for assessing model performance. Their theoretical foundations and applications have been extensively explored in areas such as information and coding theory, as seen in the works of \cite{Nat1968} and \cite{Bhat1995}. More recent developments in the study of inaccuracy measures can be found in the literature by \cite{Kun2015}, \cite{Kunb2016} and \cite{Kay2017}. In line with CRE \eqref{eq:1.2}, \cite{Kum2015} proposed a survival-based  inaccuracy measure, namely, cumulative residual inaccuracy (CRI), expressed as:
\begin{align}\label{eq:1.8}
	\mathcal K(X, Y) = -\int_0^{\infty} S_X(w)\ln S_Y(w)\,dw,
\end{align}
where $S_X$ and $S_Y$ denote the sfs of  actual and assumed models, respectively, offering a tail-focused perspective on model discrepancy.

\cite{Dan2019} introduced the weighted version of cumulative residual inaccuracy (WCRI), defined by
\begin{align}\label{1.8}
    \mathcal{K}^w(X, Y) = -\int_{0}^{\infty}w S_X(w)\ln S_Y(w)\,dw.
\end{align}
They subsequently extended this framework to record values and developed the corresponding WCRI. In their study, they established several fundamental properties and characterization results, including links with classical reliability measures, bounds, stochastic ordering relations, and the behavior of the measure under linear transformations. Dynamic versions of the measure were also formulated, and an empirical estimation approach based on lower record values was proposed and analyzed.

The WCRI  is an important measure for quantifying the discrepancy between a true lifetime model and an assumed model by incorporating cumulative information and weighting effects. However, WCRI may not adequately capture complex tail behavior and structural discrepancies between true and assumed RVs, particularly for heterogeneous and non-standard lifetime distributions. Motivated by this limitation, we propose a generalized measure, called the  weighted  fractional generalized cumulative residual inaccuracy (WFGCRI), which extends WCRI by providing a more comprehensive assessment of discrepancy between true and assumed models through their sfs. The WFGCRI is defined as:
\begin{align}\label{eq1.9}
	\mathcal K_\beta(X,Y;\psi)=\frac{1}{\Gamma(\beta+1)}\int_0^\infty \psi(w)S_X(w)(-\ln S_Y(w))^\beta\,dw,~ ~\beta\geq0.
\end{align}
It is worth noting that for identical RVs \(X\) and \(Y\), it reduces to the WFGCRE (\ref{eq::1.6}). If further \(\psi(w)=1\), it simplifies to FGCRE \eqref{eq:5}. Also for \(\beta=1\), \eqref{eq1.9} coincides with WCRI \eqref{1.8} with weight function $\psi(w)$.

Upon neglecting the weight factor, the proposed WFGCRI measure can be compared with the fractional cumulative residual inaccuracy (FCRI) introduced by \cite{Kharb2024}, which is defined as:
\begin{align}\label{eq1.11}
	\tilde{\mathcal{K}}_\beta(X,Y)=\int_0^\infty S_X(w)\big(-\ln S_Y(w)\big)^\beta\,dw, ~~\beta\geq 0.
\end{align}
The FCRI is an extension of the FCRE (\ref{1.4}). However, in  equation (\ref{eq1.11}), the parameter space of $\beta$ was incorrectly taken as $\beta \geq 0$, whereas the appropriate range should be $0 \leq \beta \leq 1$ to ensure consistency with the fractional framework. To address this issue, an appropriate corrected definition of the FCRI, which we refer to as the fractional generalized cumulative residual inaccuracy (FGCRI), defined by
\begin{align}\label{eq:1.12}
	\mathcal{K}_\beta(X,Y)=\frac{1}{\Gamma(\beta+1)}\int_0^\infty S_X(w)\big(-\ln S_Y(w)\big)^\beta\,dw, \quad \beta \geq 0.
\end{align}
Several studies have also been devoted to weighted versions of entropy and inaccuracy measures. Important contributions in this direction include \cite{Kumar2010}, \cite{Mis2011}, \cite{Kundu2014}, \cite{Bara2016}, \cite{Fak2016}, \cite{Kundu2017}, \cite{Tahm2020} and \cite{Kay2023}, which demonstrate the relevance of incorporating weight functions in the analysis of uncertainty and information measures. Our proposed measure (WFGCRI) \eqref{eq1.9} can also be considered as the weighted version of \eqref{eq:1.12}. It also extends the notion of WFGCRE \eqref{eq:5} to the case when an appropriate model is assumed in the absence of a true model. With these motivations, in this manuscript we explore the proposed WFGCRI for its properties and applications.

The rest of the paper is organized as follows. Section 2 investigates the analytical bounds of WFGCRI and examines its properties under various stochastic orders and the mixture hazard model.
Section 3 proposes a dynamical version of the WFGCRI and studies its properties under the proportional hazard rate and proportional odds models. We also develop a nonparametric estimation procedure and assess its performance through Monte Carlo simulations in Section 4.
In Section 5, we examine the behavior of WFGCRI under the proportional hazard rate model using the Ricker and Tent maps and demonstrate its usefulness as a measure of instability in time series data. Furthermore, we show that WFGCRI outperforms WCRI through an application to a real-world dataset. Finally, Section 6 concludes the paper by summarizing the main results and contributions.
	\section{Weighted Fractional Generalized Cumulative Residual Inaccuracy Measure}
	In this section, we study the WFGCRI measure and investigate its fundamental theoretical properties. We derive several analytical properties and establish meaningful upper and lower bounds that provide insight into the structural behavior of WFGCRI. Furthermore, we analyze the measure under different stochastic orderings, which helps to understand its monotonicity and comparative behavior across RVs.
	
In the following, we establish two upper bounds for the WFGCRI.
	\begin{theorem}
		For two RVs $X$ and $Y$ with sfs $S_X$ and $S_Y$, respectively. If $\mathcal K_\beta(X,Y;\psi)$ is finite, then
		\begin{enumerate}
			\item[(i)] $
			\mathcal K_\beta(X, Y;\psi) \geq \frac{1}{\Gamma(\beta+1)}\int_0^{\infty} \psi(w) S_X(w)(F_Y(w))^\beta dw
			$;
			\item[(ii)] $\mathcal K_\beta(X,Y;\psi)\geq \mathrm{e}^{\delta(X)+\frac{1}{\Gamma(\beta+1)} H(X)},$ where $\delta(X)=\mathrm{e}^{\frac{1}{\Gamma(\beta+1)}\int_0^{\infty} f_X(w)\ln\left[\psi(w)S_X(w)(-\ln(S_Y(w)))^\beta\right]dw}$.
		\end{enumerate}
	\end{theorem}
	\begin{proof}
		First part of the proof follows from a well known-inequality $\ln w\leq 1-w,~ \forall w\leq 1$. For the second part, applying log-sum inequality, we obtain
		\begin{align*}
		\frac{1}{\Gamma(\beta+1)}	\int_0^{\infty}& f_X(w)\ln\left[\frac{f_X(w)}{\psi(w)S_X(w)\left(-\ln(S_Y(w))^\beta\right)}\right]dw\geq \\	&\frac{1}{\Gamma(\beta+1)}\ln\left[\
			\frac{1}{\int_0^{\infty} \psi(w)S_X(w)\left(-\ln(S_Y(w))\right)^\beta dw}\right].
		\end{align*}
		Equivalently, 
		\begin{align*}
	\frac{1}{\Gamma(\beta+1)}&	\ln\left[\int_0^{\infty} \psi(w)S_X(w)\left(-\ln(S_Y(w))\right)^\beta dw\right]	\geq\\&\frac{1}{\Gamma(\beta+1)} {\int_0^{\infty} f_X(w)\ln\left[\psi(w)S_X(w)(-\ln(S_Y(w)))^\beta\right]dw}+\frac{1}{\Gamma(\beta+1)}{H(X)}.
		\end{align*}
		This  concludes the proof.
	\end{proof}
    Stochastic orders play a fundamental role in probability and statistics, offering a powerful framework for comparing RVs and stochastic processes across numerous disciplines. They provide robust and interpretable tools that support improved decision-making and effective risk assessment. Let $X$ and $Y$ be RVs with sfs $S_X$ and $S_Y$, respectively. 
\begin{enumerate}
	\item \textbf{Usual stochastic order.}  
	We say that $X$ is smaller than $Y$ in the usual stochastic order, denoted by $X \leq_{st} Y$, if
	$
	S_X(w) \leq S_Y(w), \quad \forall\, w \in \mathbb{R}.
	$
	
	\item \textbf{Hazard rate order.}  
	We say that $X$ is smaller than $Y$ in the hazard rate order, denoted by $X \leq_{hr} Y$, if
	$
	\frac{S_X(w)}{S_X(s)} \leq \frac{S_Y(w)}{S_Y(s)}, 
	\quad \forall\, w \leq s
	$.
\end{enumerate}
 In the following section, we derive bounds for the proposed inaccuracy measure by employing the usual stochastic order.
		\begin{theorem}
		Let $X$ and $Y$  be two non-negative RVs with sfs $S_X$ and ${S}_Y,$  respectively.  For $\beta>0$,
		\begin{enumerate}
			\item[(i)] if $X\leq_{st}Y$, then $ \mathcal K_\beta(X,Y;\psi)\leq 	\min\left\{  \tilde{\mathcal H}_\beta(X;\psi),~  \tilde{\mathcal H}_\beta(Y;\psi)\right\};$
		\item[(ii)] if $X\geq_{st}Y$, then $ \mathcal K_\beta(X,Y;\psi)\geq 	\max\left\{ \tilde{\mathcal H}_\beta(X;\psi),~ \tilde{\mathcal H}_\beta(Y;\psi)\right\},$
		\end{enumerate}
		where $\mathcal H_\beta(X;\psi)$ is defined in (\ref{eq::1.6}).
	\end{theorem}
\begin{proof}
	(i) Suppose that $X \leq_{st} Y$. Then, it follows that
	\begin{align}\label{2.2}
		\mathcal{K}_\beta(X,Y;\psi) &\leq \frac{1}{\Gamma(\beta+1)}\int_0^{\infty} \psi(w) S_Y(w) \left( -\ln S_Y(w) \right)^\beta dw =  \tilde{\mathcal H}_\beta(Y;\psi).
	\end{align}	
	Moreover, since $S_X(w) \leq S_Y(w)$, we also have $-\ln S_X(w) \geq -\ln S_Y(w)$, and thus
	$$
	\left( -\ln S_X(w) \right)^\beta \geq \left( -\ln S_Y(w) \right)^\beta \quad \text{for } \beta > 0.
	$$
	Using this, we again bound the integrand in another direction:
	\begin{align}\label{2.3}
		\mathcal{K}_\beta(X,Y;\psi) &=\frac{1}{\Gamma(\beta+1)} \int_0^{\infty} \psi(w) S_X(w) \left( -\ln\nonumber S_Y(w)\nonumber \right)^\beta dw \\
		&\leq\frac{1}{\Gamma(\beta+1)} \int_0^{\infty} \psi(w) S_X(w) \left( -\ln S_X(w) \right)^\beta dw = \tilde{\mathcal H}_\beta(X;\psi).
	\end{align}
	Combining inequalities (\ref{2.2}) and (\ref{2.3}), we obtain the desired result. The second part of the proof follows analogously and is omitted for brevity.
\end{proof}
	\begin{theorem}
		Let $X$, $Y$ and $Z$ be three RVs with sfs $S_X, S_Y$ and $S_Z$, respectively.  Then 
		\begin{enumerate}
			\item[(i)] $\mathcal K_\beta(Z,X;\psi)\geq 	\mathcal K_\beta(Z,Y;\psi)~\text{if}~X\leq_{st}Y.$
			\item[(ii)] $\mathcal K_\beta(X,Y;\psi)\geq 	K_\beta(X,Z;\psi)~\text{if}~Y\leq_{st}Z.$
		\end{enumerate}
	\end{theorem}
	\begin{proof}
(i).	If $X\leq_{st}Y$, then $S_X(w)\leq S_Y(w)$ for all $x$. Equivalently,
	\begin{align*}
S_Z(w)\left(-\ln S_X(w)\right)^\beta\geq S_Z(w)\left(-\ln S_Y(w)\right)^\beta.
	\end{align*}
	Now multiplying both sides by $\psi(w)$ and integrating both sides we obtain the desired result. A similar argument follows for the second part.
	\end{proof}
	\begin{theorem}
		Let $X,Y$ and $Z$ be three non-negative RVs so that $X\leq_{st} Y\leq_{st} Z$, then $\mathcal K_\beta(X,Y;\psi)+\mathcal K_\beta(Y,Z;\psi)\geq 2	\mathcal K_\beta(X,Z;\psi).$
	\end{theorem}
	\begin{proof}
		The proof directly follows from the definition of  WFGCRI measure.
	\end{proof}
	Next, we analyze the behavior of the WFGCRI under monotone transformations.
    \begin{theorem}
Let $X,Y$ be nonnegative absolutely continuous RVs with sfs $S_X, S_Y$, respectively.
Let $\varphi:[0,\infty)\to[0,\infty)$ be a continuously differentiable and 
strictly increasing function, and define $U=\varphi(X)$ and $V=\varphi(Y)$. Then 
\begin{align*}
	\mathcal K_\beta(U,V;\psi)&=\begin{cases}
		\int_{0}^{\infty}\psi(\varphi(w))\,S_X(w)\,
		\big(-\ln  S_Y(w)\big)^{\beta}\varphi'(w)\,dw,~~\text{if}~~\varphi~\text{is strictly increasing}\\~~~~~~~~~~~~~~~~~~~~~~~~~~~~~~~~~~~~~~~~~~~~\text{with}~\varphi^{-1}(0)=0~\text{and}~ \varphi^{-1}(\infty)=\infty;\\
			\int_{0}^{\infty}\psi(\varphi(w))\, F_X(w)\,
		\big(-\ln  F_Y(w)\big)^{\beta}|\varphi'(w)|\,dw,~~\text{if}~~\varphi~\text{is strictly decreasing}\\~~~~~~~~~~~~~~~~~~~~~~~~~~~~~~~~~~~~~~~~~~~~\text{with}~\varphi^{-1}(0)=\infty~\text{and}~ \varphi^{-1}(\infty)=0.\\
	\end{cases}
\end{align*}
\end{theorem}
\begin{proof}
Suppose that $\varphi$ is a strictly increasing function  with its  inverse $\varphi^{-1}$. For any $u\ge0$,
$S_U(u)=P(\varphi(X)>u)
= P\!\big(X>\varphi^{-1}(u)\big)
= S_X(\varphi^{-1}(u)),$
and similarly $ S_V(v)= S_Y(\varphi^{-1}(v))$.
From (\ref{eq1.9}), we have
\begin{align*}
\mathcal K_\beta(U,V;\psi)
&=\frac{1}{\Gamma(\beta+1)}\int_0^\infty 
\psi(u)\,S_U(u)\,
\big(-\ln S_V(u)\big)^{\beta}\,du \\ &=\frac{1}{\Gamma(\beta+1)}\int_0^\infty 
\psi(u)\,S_X(\varphi^{-1}(u))\,
\big(-\ln S_Y(\varphi^{-1}(u))\big)^{\beta}\,du \\
&=\frac{1}{\Gamma(\beta+1)} \int_0^\infty 
\psi(\varphi(w))\,S_X(w)\,
\big(-\ln S_Y(w)\big)^{\beta}\,
\varphi'(w)\,dw .
\end{align*} The last equality is obtained using $u=\varphi(w)$. The result can be established for the strictly decreasing case by an analogous argument.
\end{proof}
Next, we investigate the effect of affine transformations on the WFGCRI. 
	\begin{theorem}
		Suppose $X_1,X_2,Y_1$ and $Y_2$ are RVs with $Y_1=aX_1+b$ and $Y_2=aX_2+b$,~$a>0,~b\geq 0$. Then, for $\beta>0$, we have 
		\begin{align*}
			\mathcal K_\beta(Y_1,Y_2;\psi)=\alpha 	\mathcal K_\beta(X_1,X_2;\overline{\psi}),
		\end{align*}
		where $\overline \psi(w)=\psi(aw+b)$.
	\end{theorem}
	\begin{proof}
The sfs of $Y_1$ and $Y_2$ are $S_{Y_1} (w) = S_{X_1} \left(\frac{w-b}{a}\right)$ and $S_{Y_2} (w) = S_{X_2} \left(\frac{w-b}{a}\right)$, respectively. Now, from the definition of WFGCRI measure, we have
\begin{align*}
		\mathcal K_\beta(Y_1,Y_2;\psi)&=\frac{1}{\Gamma(\beta+1)}\int_0^{\infty} \psi(w)S_{Y_1}(w)\left(-\ln S_{Y_2}(w)\right)^\beta dw\\
		&=\frac{1}{\Gamma(\beta+1)}\int_b^{\infty} \psi\left(w\right)S_{X_1}\left(\frac{w-b}{a}\right)\left(-\ln S_{X_2}\left(\frac{w-b}{a}\right)\right)^\beta dw\\
		&=\frac{a}{\Gamma(\beta+1)}\int_0^{\infty} \psi(aw+b)S_{X_1}(w)\left(-\ln S_{X_2}(w)\right)^\beta dw.
\end{align*}
This establishes the proof.
	\end{proof}
In the following, we derive upper and lower bounds for WFGCRI in terms of CRI \eqref{eq:1.8}.
	\begin{theorem}
		If $X$ and $Y$ are two non-negative RVs with common support $(a,b)$, such that $0<a<b<\infty$ and $\mathcal K_\beta(X,Y)<\infty$, then 
			\begin{enumerate}
				\item [(i)]		$	\mathcal K_\beta(X, Y;\psi)\geq \inf_{w\in(a,b)}\psi(w)\frac{[\mathcal    {K}(X,Y)]^\beta}{\Gamma{(\beta+1)}(b-a)^{\beta-1}},~~\beta\geq 1;$
				\item[(ii)]	$\mathcal K_\beta(X, Y;\psi)\leq \sup_{w\in(a,b)}\psi(w)\frac{[\mathcal K(X,Y)]^\beta}{\Gamma(\beta+1)(b-a)^{\beta-1}},~~0<\beta\leq 1.$
				\end{enumerate}
	\end{theorem}	
	\begin{proof}
	(i)	For $\beta \ge 1$, $(S_X(w))^\beta \le S_X(w),$
		which implies
		\[
		\psi(w)S_X(w)\big(-\ln S_Y(w)\big)^\beta 
		\ge \psi(w)(S_X(w))^\beta\big(-\ln S_Y(w)\big)^\beta.
		\]
		Multiplying both the sides by $1/\Gamma(\beta+1)$ and integrating over $(a,b)$, we obtain
		\[
		\mathcal{K}_\beta(X,Y;\psi)
		\ge\frac{1}{\Gamma(\beta+1)} \int_a^b \psi(w)(S_X(w))^\beta\big(-\ln S_Y(w)\big)^\beta\,dw.
		\]
		Since $\psi(w)\ge \inf_{w\in(a,b)}\psi(w)$, it follows that
		\[
		\mathcal{K}_\beta(X,Y;\psi)
		\ge \inf_{w\in(a,b)}\frac{1}{\Gamma(\beta+1)}\psi(w)\int_a^b (S_X(w))^\beta\big(-\ln S_Y(w)\big)^\beta\,dw.
		\]
		Now define $\eta(w)=-S_X(w)\ln S_Y(w)\ge 0$ and $\phi_\beta(z)=z^\beta$.  
		For $\beta>1$, the function $\phi_\beta$ is convex on $[0,\infty)$.  
		By Jensen's inequality,
		\[
		\int_a^b \eta(w)^\beta\,dw
		\ge (b-a)\,\phi_\beta\!\left(\frac{1}{b-a}\int_a^b \eta(w)\,dw\right)
		= \frac{1}{(b-a)^{\beta-1}}\left(\int_a^b \eta(w)\,dw\right)^\beta.
		\]
		Therefore,
		\[
		\mathcal{K}_\beta(X,Y;\psi)
		\ge \inf_{w\in(a,b)}\psi(w)\,
		\frac{1}{(b-a)^{\beta-1}}\frac{1}{\Gamma(\beta+1)}
		\left(-\int_a^b S_X(w)\ln S_Y(w)\,dw\right)^\beta,
		\]
(ii)	Similarly, when $0 < \beta \leq 1$, the inequality
$
\big(S_X(w)\big)^\beta \geq S_Y(w)
$
holds.   Further, we handle the case by recognizing that $\phi_\beta(z) = z^\beta$ is a concave function for $z \geq 0$.  
	\end{proof}
	\begin{theorem}
		Let $X$ and $Y$ be two non-negative RVs with sfs $S_X$ and $S_Y$, respectively and $\psi(w)=\left(\zeta(w)\right)^\beta$,  $\zeta(w)\geq 0$, then we have 
		\begin{align*}
		\mathcal K_\beta(X,Y;\psi)\geq(\leq)\frac{1}{\Gamma(\beta+1)} \left(\mathcal K^{\zeta}(X,Y)\right)^\beta~\text{for}~\beta>1~(0<\beta<1),
		\end{align*}
		where $\mathcal K^\zeta(X,Y)$ is WCRI defined in \eqref{1.8} with weight function $\zeta(w)$.
	\end{theorem}
	\begin{proof}
		Let $\beta\geq 1$, then  $S_X(w)\geq \left(S_X(w)\right)^\beta$.  From (\ref{eq1.9}), we have 
		\begin{align*}
		\mathcal K_\beta(X,Y;\psi)&=\frac{1}{\Gamma(\beta+1)}\int_0^{\infty}\psi(w) S_X(w)\left(-\ln (S_Y(w))\right)^\beta dw\\
			&\geq \frac{1}{\Gamma(\beta+1)}\int_0^{\infty} \psi(w)S_X^\beta (w)\left(-\ln S_Y(w)\right)^\beta dw\\
			&=\frac{1}{\Gamma(\beta+1)} \int_0^{\infty}\left(\zeta(w)S_X(w)\big(-\ln S_Y(w)\big)\right)^\beta \,dw.
		\end{align*}
		Let $\Delta(w)=\zeta(w)S_X(w)\left(-\ln S_Y(w)\right)$ and $g(\Delta(w))=\left(\zeta(w)S_X(w)\left(-\ln S_Y(w)\right)\right)^\beta$. Since $g(z)=z^\beta$ is convex in $z\geq 0$, for $\beta\geq 1$, using Jensen's inequality
		we obtain the desired result. 
	\end{proof}
 \subsection{WFGCRI under Mixture Hazard Rate Model}
The mixture hazard model provides a flexible framework for representing heterogeneous failure mechanisms by combining a finite number of underlying hazard rate functions. This formulation
captures variability arising from latent risk sources and allows the overall system to
exhibit behavior that neither component model can achieve on its own.
Let $r_1(\cdot), r_2(\cdot), \dots, r_k(\cdot)$ be $k$ hazard rate functions and let
$p_1,\dots,p_k$ be positive weights such that $p_i>0,~~ \sum_{i=1}^k p_i = 1.$
The {mixture hazard rate model} is defined by
$	r_p(w)=\sum_{i=1}^k p_i\, r_i(w), ~ w>0.$ The sf corresponding to $r_p(\cdot)$ is given by
\begin{align}\label{eq:2.3}
	S_p(w)
	&=\exp\!\left(-\int_0^{w} r_p(t)\,dt\right) \notag\\
	&=\exp\!\left(-\sum_{i=1}^k p_i \int_0^{w} r_i(t)\,dt\right) \notag\\
	&=\prod_{i=1}^k
	\exp\!\left(-\int_0^{w} r_i(t)\,dt\right)^{p_i}.
\end{align}
Since $S_i(w)=\exp\!\left(-\int_0^{w} r_i(t)\,dt\right)$, it follows that
$S_p(w)=\prod_{i=1}^k S_i(w)^{p_i},$
which is the natural extension of the classical two-component mixture hazard rate model.

\medskip

The following theorem establishes an upper bound for the WFGCRI of the mixture
in terms of the corresponding WFGCRI measures of the components.

\begin{theorem}\label{thm:kmixtureLWRI}
	For the mixture hazard model defined in \eqref{eq:2.3}, the WFGCRI measure satisfies the upper bound
	$$
	\mathcal{K}_\beta(S_p, S_Y;\psi)
	\;\le\;
	\sum_{i=1}^k p_i\, \mathcal{K}_\beta(S_i, S_Y;\psi).
	$$
\end{theorem}

\begin{proof}
	From the definition of the WFGCRI  measure, we have
	\begin{align*}
		\mathcal{K}_\beta(S_p, S_Y;\psi)
		&=
		\frac{1}{\Gamma(\beta+1)}
		\int_0^{\infty}
		\psi(w)
		\prod_{i=1}^k S_i(w)^{p_i}
		\bigl(-\ln S_Y(w)\bigr)^\beta
		\,dw.
	\end{align*}
	By the weighted arithmetic mean–geometric mean inequality,
	\[
	\prod_{i=1}^k S_i(w)^{p_i}
	\;\le\;
	\sum_{i=1}^k p_i S_i(w),
	\qquad w>0.
	\]
	Therefore,
	\begin{align*}
		\mathcal{K}_\beta(S_p, S_Y;\psi)
		&\le
		\frac{1}{\Gamma(\beta+1)}
		\int_0^{\infty}
		\psi(w)
		\left(
		\sum_{i=1}^k p_i S_i(w)
		\right)
		\bigl(-\ln S_Y(w)\bigr)^\beta
		\,dw\\
		&=
		\sum_{i=1}^k p_i
		\frac{1}{\Gamma(\beta+1)}
		\int_0^{\infty}
		\psi(w)
		S_i(w)
		\bigl(-\ln S_Y(w)\bigr)^\beta
		\,dw\\
		&=
		\sum_{i=1}^k p_i\, \mathcal{K}_\beta(S_i, S_Y;\psi).
	\end{align*}
		Equality  holds if and only if
	$S_1(w)=S_2(w)=\cdots=S_k(w)$ for almost all $w>0$.
\end{proof}
An illustration of Theorem 2.9 is provided by the following example.
\begin{exe}
	Let $S_1,S_2$ and $S_3$ be exponential sfs with
	$S_i(w)=\exp(-\lambda_i w)$, where $(\lambda_1,\lambda_2,\lambda_3)=(1.2,1.5,2.5)$.
	Let $(p_1,p_2,p_3)=(0.3,0.4,0.3)$.
	Under the three-component mixture hazard model
	$r_p(w)=\sum_{i=1}^3 p_i\lambda_i$,
	the induced sf is
	$S_p(w)=\prod_{i=1}^3 S_i(w)^{p_i}$. Take $S_Y(w)=\exp(-w)$, $\psi(w)=w$, and $\beta>0$. Figure~\ref{fig:1} plots
	$\mathcal{K}_\beta(S_p,S_Y;\psi)$ and
	$\sum_{i=1}^3 p_i\mathcal{K}_\beta(S_i,S_Y;\psi)$ as functions of $\beta$,
	showing that the WFGCRI of the mixture is uniformly dominated by the
	corresponding convex combination, in agreement with Theorem~2.9.
	\begin{figure}
		\centering
		\includegraphics[width=5in, height=2.5in]{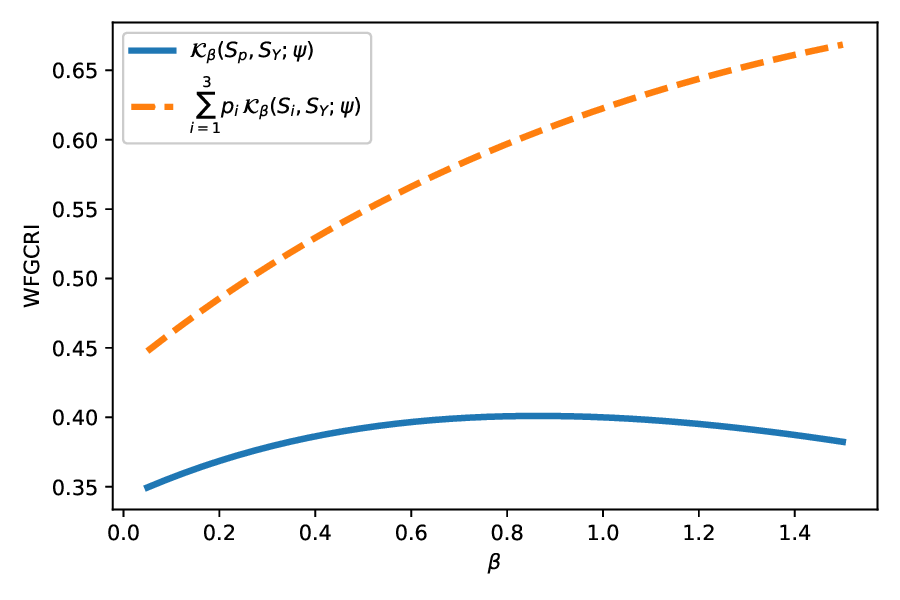}
		\caption{Graphical illustration of the WFGCRI for a mixture of three exponential distributions with different parameter settings (Example 2.1).}\label{fig:1}
	\end{figure}
\end{exe}
\section{Dynamic Weighted Fractional Generalized Cumulative Residual Inaccuracy}
In this section, we introduce a dynamical formulation of the WFGCRI motivated by reliability theory. Consider a system that begins operation at time $0$ and is observed to be functioning at a fixed inspection time $t>0$. Thus, the system is known to have survived up to time $t$, while its exact failure time remains unknown. Unlike the past lifetime setting, the uncertainty here concerns the future, since failure occurs after the inspection time. Accordingly, the failure time lies in $(t,\infty)$, and the analysis is formulated in terms of the residual lifetime, that is, the remaining lifetime conditional on survival up to time $t$.
Now, we introduce the time-dependent version of WFGCRI, which is known as dynamic  weighted  fractional generalized cumulative residual inaccuracy (DWFGCRI).  Let $S_X$ and $S_Y$ be  sfs of the RVs $X$ and $Y$, respectively. Then, the  DWFGCRI measure is defined as: 
\begin{align}\label{eq3.1}
\mathcal K^t_\beta(X,Y;\psi)=\frac{1}{\Gamma(\beta+1)}\int_t^{\infty}\psi(w)\frac{S_X(w)}{S_X(t)}\left[-\ln \frac{S_Y(w)}{S_Y(t)}\right]^\beta dw,~~ \beta>0, ~t>0.
\end{align}
When $t=0$, DWFGCRI coincides with the WFGCRI.

 An illustrative example of DWFGCRI is presented below.
\begin{exe}
Consider exponential and Rayleigh distributions with sfs
\( S_X(w)=e^{-aw}, \, a>0 \), and \( S_Y(w)=e^{-(bw)^2}, \, b>0\),
respectively. We present numerical plots of DWFGCRI for
different parameter values using the weight functions \( \psi(w)=w \)
and \( \psi(w)=w^2 \), as shown in Figures~\hyperref[fig2]{2(a)} and~\hyperref[fig2]{2(b)}, respectively.
It is evident from Figure~\hyperref[fig2]{2}  that DWFGCRI exhibits non-monotonic behavior with respect to $t$.
    		\begin{figure}[H] 
			\centering
			\begin{minipage}[b]{0.46\linewidth}
				\includegraphics[height=6cm,width=8.5cm]{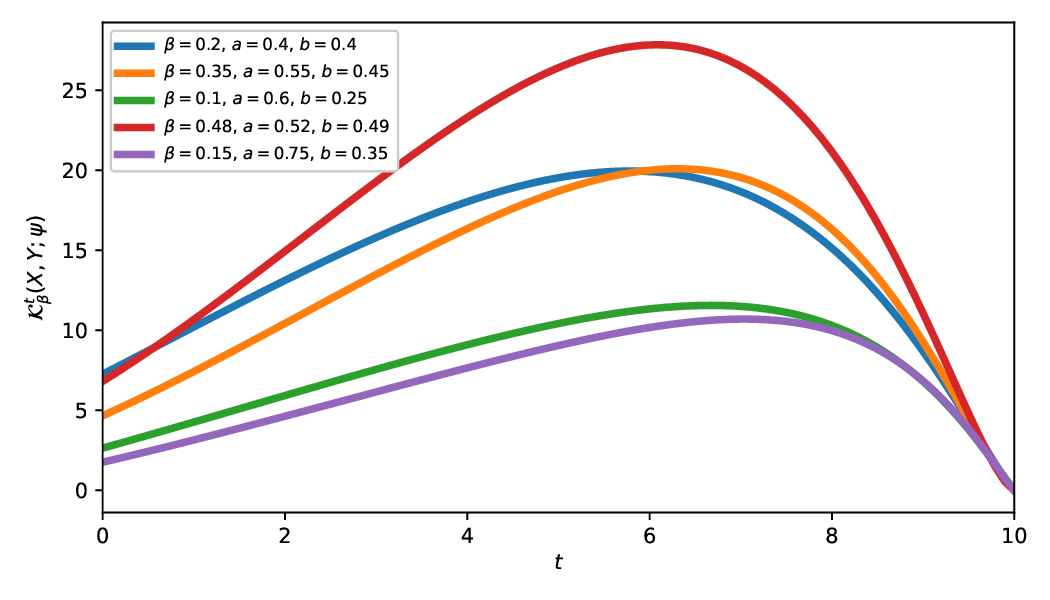}
				\centering{(a)}
			\end{minipage}
			\quad
			\begin{minipage}[b]{0.46\linewidth}
				\includegraphics[height=6cm,width=8.5cm]{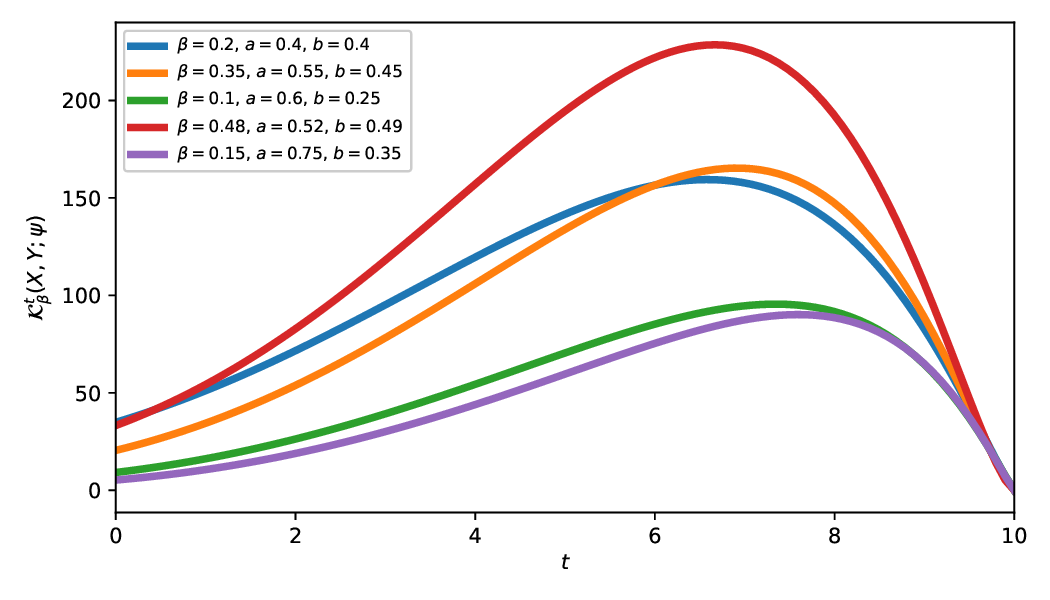}
				\centering{(b)}
			\end{minipage}
			\caption{Plots of DWFGCRI for exponential and Rayleigh distributions  for different parameter values and weight functions: (a) \(\psi(w)=w\) and (b) \(\psi(w)=w^2\) (Example 3.1).
			}\label{fig2}
		\end{figure}
\end{exe}
Similar to  Theorem 2.6, we get the following result concerning the effect of an affine transformation. The proof is omitted for brevity.  
\begin{theorem}
		Suppose $X_1,X_2,Y_1$ and $Y_2$ are RVs with $Y_1=aX_1+b$ and $Y_2=aX_2+b$, $a>0,b\geq 0$. Then, we have 
	\begin{align*}
		\mathcal K^t_\beta(Y_1,Y_2;\psi)=\frac{a}{\Gamma(\beta+1)}	\mathcal K^{\frac{t-b}{a}}_\beta\left(X_1,X_2;{\overline{\psi}}\right),
	\end{align*}
	where $\overline \psi(w)={\psi(aw+b)}$.
\end{theorem}
\subsection{DWFGCRI under Proportional Hazard Rate Model}
 The RVs $X$ and $Y$ are said to satisfy the proportional hazard rate (PHR) model if there exists a positive constant $\alpha > 0$ such that
$	r_Y(w) = \alpha \, r_X(w),~ w \ge 0,$
where $r_X(\cdot)$ and $r_Y(\cdot)$ denote the hazard rate functions of $X$ and $Y$, respectively.  
Equivalently, in terms of the sfs $S_X(\cdot)$ and $S_Y(\cdot)$, the PHR model can be expressed as $	S_Y(w) = \big(S_X(w)\big)^{\alpha},~ w \ge 0.$
This model was introduced in the seminal work of \cite{cox1972}.
 The PHR model is widely used in survival analysis, particularly in medical statistics and reliability engineering, where it captures the effect of covariates (such as treatment, environment, or subject characteristics) on the lifetime of an individual or system.

	Here we study the DWFGCRI between two PHR models. Consider two PHR models $U$ and $V$ with sfs $[S_X(w)]^\alpha$ and $[S_Y(w)]^\alpha$, respectively. Then, DWFGCRI between these two models is defined as:
	\begin{align}\label{eq:2}
		\mathcal K^t_\beta(U,V;\psi)=\int_t^\infty\psi(w)\left[\frac{S_X(w)}{S_X(t)}\right]^\alpha\left[-\ln\left(\frac{S_Y(w)}{S_Y(t)}\right)^\alpha\right]^\beta\,dx.
	\end{align}
	The measure defined in (\ref{eq:2}) provides inaccuracy between two PHR models. It is not hard to see that when $t= 0$, (\ref{eq:2}) reduces to
	\begin{align*}
		\mathcal K_\beta(U,V;\psi)=\int_0^\infty\psi(w)\left[{S_X(w)}{}\right]^\alpha\left[-\ln\left({S_Y(w)}{}\right)^\alpha\right]^\beta\,dx,
	\end{align*}which is WFGCRI between two PHR models.
	\begin{exe}
Let $X$ and $Y$ follow Weibull distributions with a common shape parameter of $2$. Their respective sfs are given by
\begin{equation*}
	S_X(w) = e^{-\eta_1 w^2} \quad \text{and} \quad S_Y(w) = e^{-\eta_2 w^2}, \quad w \geq 0,
\end{equation*}
where $\eta_1, \eta_2 > 0$ represent the scale parameters. The DWFGCRI measure for the PHR model with $\psi(w)=w$  in
		(\ref{eq:2}) is obtained as
		\begin{align*}
			\mathcal K^t_\beta(U,V;\psi)&=\int_t^\infty w\left[\frac{e^{-\eta_1 w^2}}{e^{-\eta_1 t^2}}\right]^\alpha\left[-\ln\left(\frac{e^{-\eta_2 w^2}}{e^{-\eta_2 t^2}}\right)^\alpha\right]^\beta\,dw\\
			&= \frac{\eta_2^\beta}{2 \alpha \eta_1^{\beta + 1}}.
		\end{align*}
	\end{exe}
\begin{theorem}
Assume two non-negative RVs $X$ and $Y$ with sfs $S_X$ and $S_Y$, respectively. For two RVs $U$ and $V$ satisfying the PHR model, we obtain the following
\begin{align*}
\mathcal{K}^t_\beta(U,V;\psi)\leq(\geq)\alpha^\beta\mathcal{K}^t_\beta(X,Y;\psi)~~\text{if}~~\alpha>1~(0<\alpha<1). 
\end{align*}
\end{theorem}
\begin{proof}
First, we prove for  $\alpha>1$. From \eqref{eq3.1}, we have
\begin{align}\label{eq3.5}
	\mathcal{K}^t_\beta(U,V;\psi)
	&=\frac{1}{\Gamma(\beta+1)}\int_t^{\infty}\psi(w)\,
\left[	\frac{S_X(w)}{S_X(t)}\right]^\alpha\nonumber
	\left[-\ln\left(\frac{S_Y(w)}{S_Y(t)}\right)^\alpha\right]^\beta\,dw\\
	&=\frac{\alpha^\beta}{\Gamma(\beta+1)}\int_t^{\infty}\psi(w)\,
	\left[	\frac{S_X(w)}{S_X(t)}\right]^\alpha
	\left[-\ln\left(\frac{S_Y(w)}{S_Y(t)}\right)\right]^\beta\,dw .
\end{align}
For all $w\ge t$ we have $0\le \frac{S_X(w)}{S_X(t)} \le 1.$
If $\alpha\ge 1$, then for all $x\in[0,1]$, $x^\alpha\le x$. Hence, $\left(\frac{S_X(w)}{S_X(t)}\right)^\alpha
\le
\frac{S_X(w)}{S_X(t)},~ w\ge t.$
Therefore, using (\ref{eq3.5}) yields
\begin{align*}
	\mathcal{K}^t_\beta(U,V;\psi)\le \frac{\alpha^\beta}{\Gamma(\beta+1)}\int_t^{\infty}\psi(w)
	\frac{S_X(w)}{S_X(t)}
	\left(-\ln\left(\frac{S_Y(w)}{S_Y(t)}\right)\right)^\beta dw.
\end{align*}
Now assume that $0<\alpha<1$. Since $0\le S_X(w)/S_X(t)\le 1$ for all $w\ge t$,
we have $\left(\frac{S_X(w)}{S_X(t)}\right)^\alpha
\ge \frac{S_X(w)}{S_X(t)}.$
Thus, from (\ref{eq3.5}), we have
\begin{align*}
	\mathcal{K}^t_\beta(U,V;\psi)\ge \frac{\alpha^\beta}{\Gamma(\beta+1)}
	\int_t^{\infty}\psi(w)
	\frac{S_X(w)}{S_X(t)}
	\left(-\ln\left(\frac{S_Y(w)}{S_Y(t)}\right)\right)^\beta dw.
\end{align*} Hence the result follows.
\end{proof}
\subsection{DWFGCRI under Proportional Odds Model}
The proportional odds (PO) model provides an important alternative to the proportional hazards framework and has been widely used in survival and reliability analysis due to its flexibility in describing lifetime distributions. We say that RVs $X$ and $Y$ satisfy the PO model if
$\Phi_Y(t)=\alpha \Phi_X(t), ~\alpha>0,$
where $\Phi_X(t)=\frac{S_X(t)}{F_X(t)}$ and $\Phi_Y(t)=\frac{S_Y(t)}{F_Y(t)}$. Equivalently, the sfs are related by $S_Y(t)=\frac{\alpha S_X(t)}{1-\bar{\alpha} S_X(t)}, ~~ \bar{\alpha}=1-\alpha.$
 Consider the RVs $U$ and $V$ with respective odds functions
$\Phi_X(t)$ and $\Phi_Y(t)$. For further details on PO model, one may refer to \cite{kir2001}.

Now, we  formulate the DWFGCRI within the framework of PO model. The DWFGCRI measure associated with the RVs $U$ and $V$ is obtained as:
\begin{align}\label{eq:3.4}
	\mathcal{K}^t_\beta(U,V;\psi)
	&=
	\frac{1}{\Gamma(\beta+1)}
	\int_t^{\infty}
	\psi(w)\,\nonumber
	\frac{\dfrac{\alpha S_X(w)}{1-\bar{\alpha} S_X(w)}}{\dfrac{\alpha S_X(t)}{1-\bar{\alpha} S_X(t)}}
	\left[
	\ln\!\left(
	\frac{\dfrac{\alpha S_Y(t)}{1-\bar{\alpha} S_Y(t)}}
	{\dfrac{\alpha S_Y(w)}{1-\bar{\alpha} S_Y(w)}}
	\right)
	\right]^{\beta}
	\,dw\\
	&=\frac{1}{\Gamma(\beta+1)}
	\int_t^{\infty}
	\psi(w)\,
	\frac{S_X(w)}{S_X(t)}
	\frac{1-\bar{\alpha} S_X(t)}{1-\bar{\alpha} S_X(w)}
	\left[
	\ln\!\left(\frac{S_Y(t)}{S_Y(w)}\right)
	+
	\ln\!\left(\frac{1-\bar{\alpha} S_Y(w)}{1-\bar{\alpha} S_Y(t)}\right)
	\right]^{\beta}
	\,dw.
\end{align}

To elucidate the implications of equation~\eqref{eq:3.4}, we present the following illustrative example.
\begin{exe}
	Let $X$ and $Y$ be two nonnegative RVs with sfs given by
$S_X(w) = (1+w)e^{-w}$ and
$S_Y(w) = e^{-2w}, ~w>0$ respectively. Then, on using (\ref{eq:3.4}), we have
\begin{align*}
&	\mathcal K^t_\beta(U,V;\psi)\\&~~~~~~=\frac{1}{\Gamma(\beta+1)}
	\int_t^{\infty}
	\psi(w)\,
	\left[\frac{(1+w)e^{-w}}{(1+t)e^{-t}}\right]
	\frac{1-\bar{\alpha} (1+t)e^{-t}}{1-\bar{\alpha} (1+w)e^{-w}}
	\left[
	2(w-t)
	+
	\ln\!\left(\frac{1-\bar{\alpha} e^{-2w}}{1-\bar{\alpha} e^{-2t}}\right)
	\right]^{\beta}
	\,dw.
\end{align*}
Since the explicit expression is analytically intractable, we numerically evaluate and plot the DWFGCRI for $\alpha = 0.5$ and different values of the parameter $\beta$, using the weight functions $\psi(w)=w^{0.3}$ and $\psi(w)=w^{2.4}$. The results are displayed in Figure~\ref{fig2*}. 
\end{exe}
    		\begin{figure}[H] 
	\centering
	\begin{minipage}[b]{0.46\linewidth}
		\includegraphics[height=6cm,width=8.5cm]{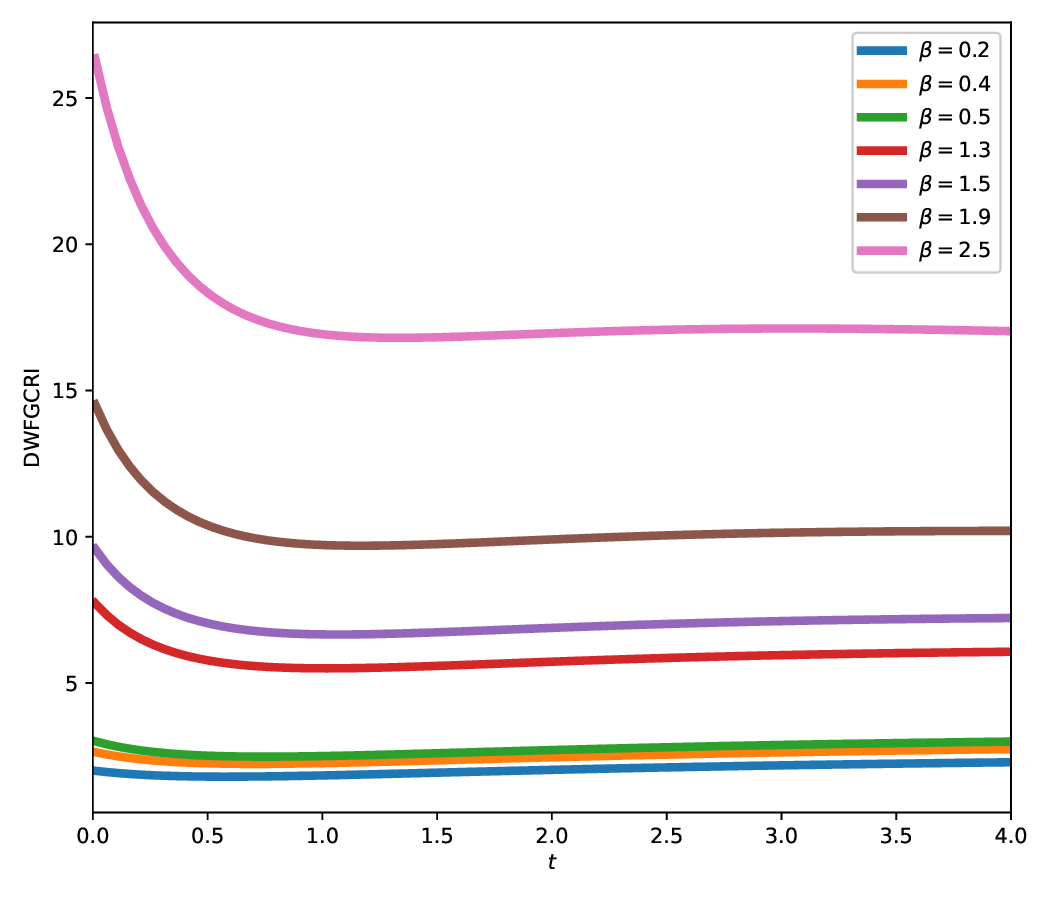}
		\centering{(a)}
	\end{minipage}
	\quad
	\begin{minipage}[b]{0.46\linewidth}
		\includegraphics[height=6cm,width=8.5cm]{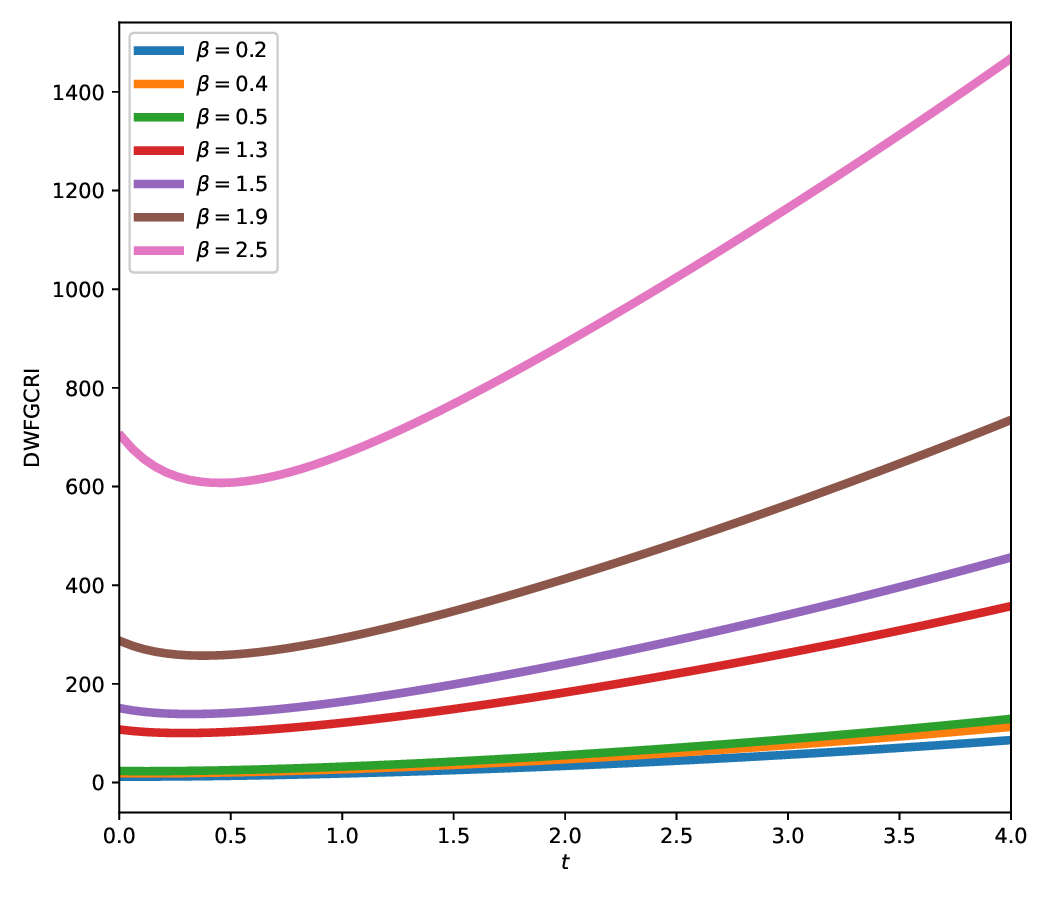}
		\centering{(b)}
	\end{minipage}
\caption{Graphical representations of  DWFGCRI  under PO model with two different weight functions: (a) $\psi(w)=w^{0.3}$ and (b) $\psi(w)=w^{2.4}$ (Example 3.3).}
\label{fig2*}
\end{figure}
\section{Non-parametric Estimation}
In this section, first we  empirically estimate the WFGCRI under the PHR model. Suppose a random sample \(X_1,\ldots,X_n\) is drawn from a population with sf \(S_X\). 
The order statistics corresponding to the sample are denoted by
$
X_{1:n} \le \cdots \le X_{n:n},$ where
$X_{1:n} = \min(X_1,\ldots,X_n),$ and 
$X_{n:n} = \max(X_1,\ldots,X_n).
$
The empirical sf of \(S_X\) is given by
\begin{align}\label{4.1}
	\widehat S_X(w) 
	=
\begin{cases}
	1 &~\text{if}~ w < X_{1:n}, \\
	1-\dfrac{j}{n} &~\text{if}~ X_{j:n} \le w < X_{j+1:n},\quad j=1,\ldots,n-1, \\
	0 & ~\text{if}~w \ge X_{n:n}.
\end{cases}
\end{align}
For the RVs \(X\) and \(Y\) with sfs
\(S_X\) and \(S_X^\alpha\) (\(\alpha>0\)), respectively, the empirical WFGCRI measure for the 
weight function \(\psi(w)=w\)  is given by
\begin{align}\label{eq4.2}
    \widehat {\mathcal K}_\beta\!\left(\widehat S_X,\,\widehat S_X^\alpha; \psi\right)\nonumber
&=
\frac{{1}}{\Gamma(\beta+1)}\int_{0}^{\infty}
w\,\widehat S_X(w)
\Bigl(-\ln\bigl(\widehat S_X(w)\bigr)^\alpha\Bigr)^{\beta}\,dw\\\nonumber
&=
\frac{{\alpha^\beta}}{\Gamma(\beta+1)}\sum_{j=1}^{n-1}
\int_{X_{j:n}}^{X_{j+1:n}}
w\,\widehat S_X(w)
\bigl[-\ln(\widehat S_X(w))\bigr]^{\beta}\,dw\\
&=\frac{{\alpha^\beta}}{\Gamma(\beta+1)}
\sum_{j=1}^{n-1}
\frac{X_{j+1:n}^{2}-X_{j:n}^{2}}{2}
\left(1-\frac{j}{n}\right)
\left[-\ln\!\left(1-\frac{j}{n}\right)\right]^{\beta}.
\end{align}
It is worth noting that
$
\mathcal K_\beta\!\left(\widehat S_X,\,\widehat S_X^\alpha; \psi\right)
$
is exactly the WFCRE with coefficient \(\alpha^\beta\).

Now consider two independent non-negative samples
$X_1,\ldots,X_n$ and
$Y_1,\ldots,Y_m,$
with empirical sfs
$\widehat S_X(w)=\frac{1}{n}\sum_{i=1}^n \mathbf{1}\{X_i>w\}$ and
$\widehat S_Y(w)=\frac{1}{m}\sum_{j=1}^m \mathbf{1}\{Y_j>w\}.$
Let $0 = t_0 < t_1 < \cdots < t_k$
be the ordered grid formed by the distinct observed values in 
\(\{X_1,\ldots,X_n,Y_1,\ldots,Y_m\}\).
Since both empirical sfs are constant on each interval 
\([t_k,t_{k+1})\), the plug--in estimator of WFGCRI is given by
\begin{align}\label{4.3}
    \widehat{\mathcal K}_\beta(X,Y;\psi)\nonumber
&=\frac{{1}}{\Gamma(\beta+1)} \int_0^{\infty}w
\widehat S_X(w)\,
\bigl(-\ln(\widehat S_Y(w))\bigr)^{\beta}\,dw\\\nonumber
&=\frac{{1}}{\Gamma(\beta+1)} \sum_{k=0}^{K-1}
\widehat S_X(t_k)\,
\Bigl(-\ln\bigl(\widehat S_Y(t_k)\bigr)\Bigr)^{\beta}
\int_{t_k}^{t_{k+1}} w\,dw\\
&=\frac{{1}}{\Gamma(\beta+1)} \sum_{k=0}^{K-1}
\widehat S_X(t_k)\,
\Bigl(-\ln\bigl(\widehat S_Y(t_k)\bigr)\Bigr)^{\beta}
\frac{t_{k+1}^2-t_k^2}{2}.
\end{align}
\subsection{Simulation Study}
In this subsection, first we assess the performance of the estimator defined in  (\ref{eq4.2}) under the PHR model. Data are generated from an exponential distribution with rate parameter $\lambda = 0.8$. Using Monte Carlo simulations, we examine the behavior of the estimator across varying sample sizes: $n = 100, 300, 500, 700$ and $1000$ with $10,000$ replications. The results, including absolute bias (AB), root mean squared error (RMSE), 95\% confidence interval (CI) length, and the mean of the estimated values, together with the true value, are summarized in Table~\ref{tab1}.

To study the effect of the tuning parameter $\beta$, we evaluate the estimator for $\alpha=0.5$ and a range of values: $\beta = 0.2, 0.5, 0.7, 0.9, 1.3$ and $1.5$. This analysis allows us to investigate the robustness and performance of the estimator under varying degrees of influence controlled by $\beta$. The results demonstrate that the estimator performs consistently across all tested values of $\beta$. As the sample size increases, the AB, RMSE, and the average length of the CI all decrease, further confirming the consistency and improved precision of the estimator.
This outcome aligns with theoretical expectations and affirms the reliability of the proposed estimator, particularly under larger sample sizes and different values of $\beta$.

Figure~\hyperref[fig3]{4(a)-(d)} illustrates the  performance of the estimator for $\beta= 0.5$ across different sample sizes. The boxplot in Figure~\hyperref[fig3]{4(a)} shows that the variability of the estimates decreases steadily as $n$ increases, indicating improved stability. Figure~\hyperref[fig3]{4(b)} presents the mean estimates along with the corresponding confidence bounds, demonstrating clear convergence towards the theoretical value. The RMSE plot in Figure~\hyperref[fig3]{4(c)} further confirms a monotone decline in estimation error with increasing sample size, supporting the consistency of the estimator. Finally, the histogram in Figure~\hyperref[fig3]{4(d)}, corresponding to the largest sample size, shows a strong concentration of the estimates around the true  value, reflecting satisfactory finite-sample performance.
\begin{table}[!htbp]
	\centering
	\caption{AB, RMSE, 95\% CI length, mean of estimated  and true values of the WFGCRI under  PHR model for different values of $\beta$.}
	\label{tab1}
	\begin{tabular}{ccccccccc}
		\toprule
		$\beta$ & $n$ & AB & RMSE & 95\% CI length & ${\widehat{\mathcal K}}_\beta(X,Y;\psi)$ & ${\mathcal K_\beta}(X,Y;\psi)$ \\
		\toprule
		&100 &0.223415&0.533088&1.840806&1.855697&1.632282\\
		&300 &0.122429&0.284783&1.011046&1.754712&1.632282\\
		$\beta$=0.2&500 &0.089776&0.219095&0.765248&1.722058&1.632282\\
		&1000&0.055833&0.143709&0.533620&1.688115&1.632282\\
		\hline
		
		&100 &0.201601&0.571568&2.029858&1.858882&1.657282\\
		&300 &0.132203&0.343172&1.234686&1.789485&1.657282\\
		$\beta$=0.5&500 &0.085233&0.239779&0.896569&1.742515&1.657282\\
		&1000&0.057632&0.165144&0.595365&1.714913&1.657282\\
		\hline
		
		&100 &0.219689&0.634015&2.248742&1.854802&1.635114\\
		&300 &0.135493&0.352373&1.303289&1.770607&1.635114\\
		$\beta$=0.7&500 &0.102994&0.278984&0.996530&1.738108&1.635114\\
		&1000&0.065143&0.178544&0.653552&1.700257&1.635114\\
		\hline
		
		&100 &0.140001&0.578307&2.180485&1.730915&1.590914\\
		&300 &0.125696&0.359799&1.305886&1.716610&1.590914\\
		$\beta$=0.9&500 &0.095500&0.292826&1.015532&1.686414&1.590914\\
		&1000&0.077364&0.197120&0.709055&1.668278&1.590914\\
		\hline
		
		&100 &0.072155&0.576120&2.203513&1.531671&1.459516\\
		&300 &0.084750&0.385335&1.434364&1.544266&1.459516\\
		$\beta$=1.3&500 &0.088506&0.309925&1.178165&1.548022&1.459516\\
		&1000&0.064822&0.220991&0.785252&1.524338&1.459516\\
		\hline
		
		&100 &0.044472&0.594980&2.337166&1.425540&1.381068\\
		&300 &0.056224&0.371652&1.423099&1.437292&1.381068\\
		$\beta$=1.5&500 &0.086754&0.310450&1.164369&1.467822&1.381068\\
		&1000&0.065220&0.223847&0.848195&1.446287&1.381068\\
		\bottomrule
	\end{tabular}
\end{table}

Furthermore, we assess the performance of the proposed estimator for the WFGCRI given in~(\ref{4.3}), which is not under the PHR model. Random samples
$X_1,\ldots,X_n$ are generated from the true underlying distribution, taken to be an exponential distribution with rate parameter $\lambda_1 = 2.5$. Independently, samples
$Y_1,\ldots,Y_n$ are generated from an assumed exponential distribution with rate parameter $\lambda_2 = 3.5$. For each sample size, based on 10{,}000 Monte Carlo replications, the AB, RMSE, length of the 95\% CI, and the mean of the estimated values, along with the corresponding true value, are computed for different values of $\beta$ and reported in Table~\ref{tab2}.
It is evident from Table~\ref{tab2}   that increasing the sample size leads to a systematic reduction in AB, RMSE, and CI length, thereby demonstrating both the consistency of the estimator and the gain in estimation accuracy and precision.    
	\begin{figure}[!htbp]
	\centering
	\begin{minipage}[b]{0.45\linewidth}
		\includegraphics[height=7cm,width=8.3cm]{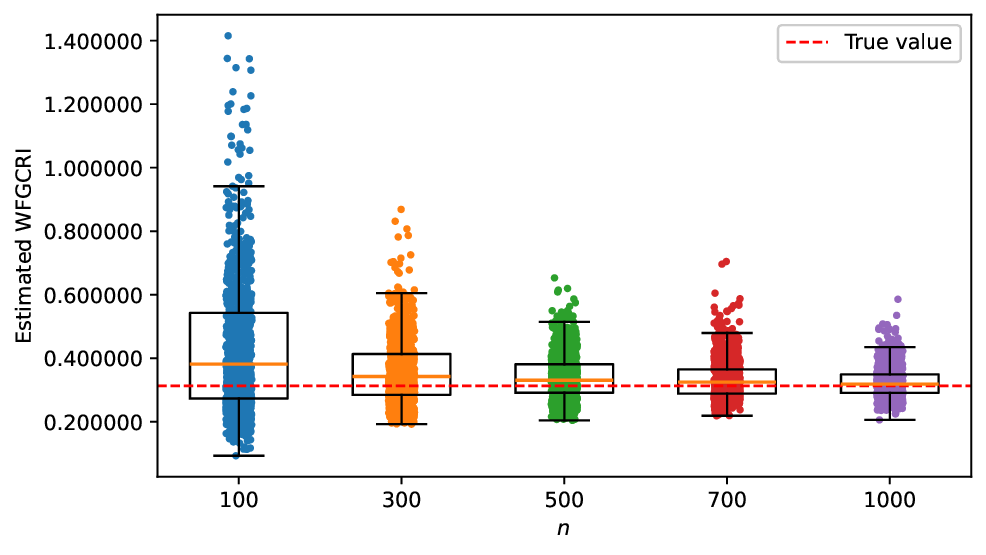}
		\centering{(a)  }
		
	\end{minipage}
	\quad
	\begin{minipage}[b]{0.45\linewidth}
		\includegraphics[height=7cm,width=8.3cm]{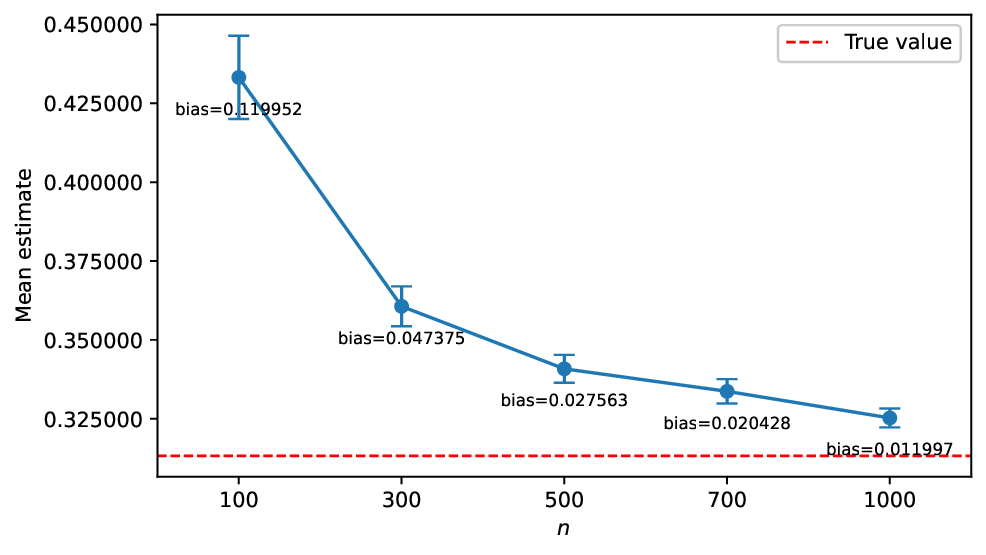}
		\centering{(b)}
	\end{minipage}
	\quad
	\begin{minipage}[b]{0.45\linewidth}
		\includegraphics[height=7cm,width=8.3cm]{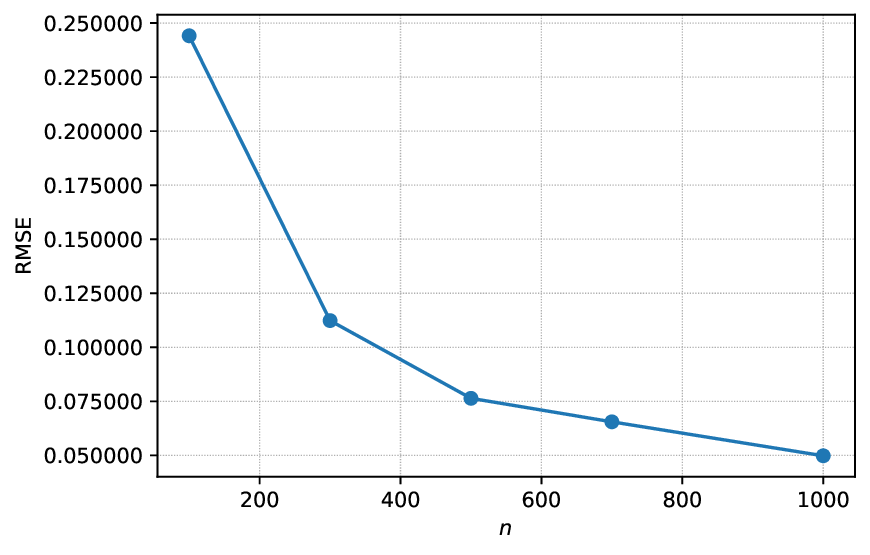}
		\centering{(c)}
	\end{minipage}
	\quad
	\begin{minipage}[b]{0.45\linewidth}
		\includegraphics[height=7cm,width=8.3cm]{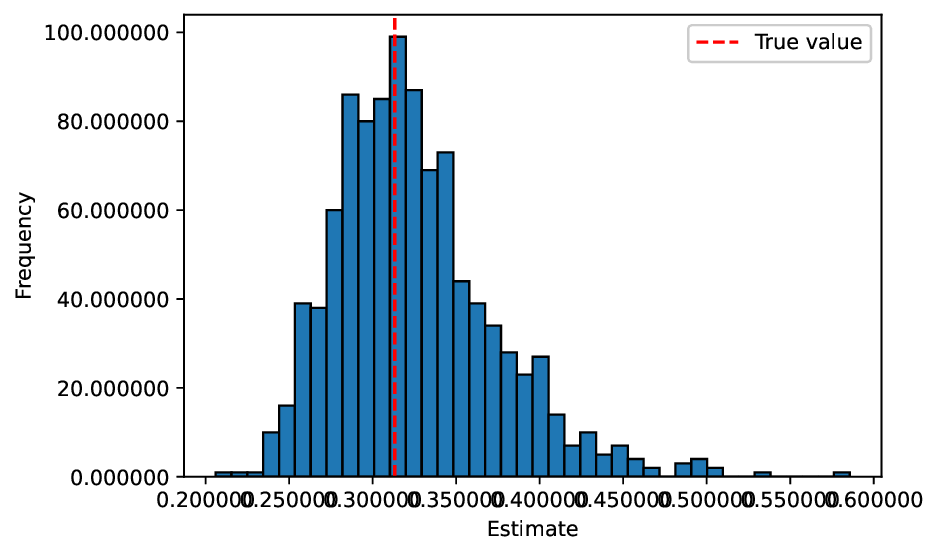}
		\centering{(d)}
	\end{minipage}
	\caption{Graphical presentation of simulation study of WFGCRI.}\label{fig3}
\end{figure}
\begin{table}[!htbp]\FloatBarrier
	\centering
	\caption{AB, RMSE, 95\% CI length, mean of estimated  and true values of the WFGCRI for different values of $\beta$.}\label{tab2}
	\label{tab:new}
	\begin{tabular}{cccccccc}
		\toprule
		$\beta$ & $n$ & AB & RMSE & 95\% CI length & $\widehat{\mathcal K}_\beta(X,Y;\psi)$ & ${\mathcal K_\beta}(X,Y;\psi)$ \\
		\toprule
		&100 &0.022318&0.084633&0.010125&0.252410&0.230092\\
		&300 &0.005960&0.042428&0.005210&0.236053&0.230092\\
		$\beta$=0.3&500 &0.003579&0.032070&0.003953&0.233671&0.230092\\
		&700 &0.001953&0.026429&0.003269&0.232045&0.230092\\
		&1000&0.000308&0.021441&0.002659&0.230400&0.230092\\
		\hline
		
		&100 &0.068760&0.166100&0.018752&0.352732&0.283972\\
		&300 &0.029253&0.086826&0.010139&0.313225&0.283972\\
		$\beta$=0.5&500 &0.019671&0.061898&0.007279&0.303643&0.283972\\
		&700 &0.013721&0.048462&0.005764&0.297693&0.283972\\
		&1000&0.006507&0.038591&0.004718&0.290479&0.283972\\
		\hline
		
		&100 &0.181358&0.333775&0.034752&0.525595&0.344238\\
		&300 &0.074526&0.155580&0.016938&0.418763&0.344238\\
		$\beta$=0.7&500 &0.045209&0.111737&0.012673&0.389447&0.344238\\
		&700 &0.032887&0.086644&0.009942&0.377125&0.344238\\
		&1000&0.025160&0.072082&0.008378&0.369398&0.344238\\
		\hline
		
		&100 &0.377731&0.650282&0.065649&0.789249&0.411518\\
		&300 &0.177643&0.318696&0.032816&0.589161&0.411518\\
		$\beta$=0.9&500 &0.109396&0.216690&0.023198&0.520914&0.411518\\
		&700 &0.083287&0.170201&0.018409&0.494805&0.411518\\
		&1000&0.054722&0.127749&0.014317&0.466240&0.411518\\
		\hline
		
		&100 &1.097496&1.726821&0.165347&1.624600&0.527104\\
		&300 &0.488480&0.789046&0.076852&1.015584&0.527104\\
		$\beta$=1.2&500 &0.316686&0.546199&0.055193&0.843790&0.527104\\
		&700 &0.224271&0.397705&0.040734&0.751375&0.527104\\
		&1000&0.171100&0.322088&0.033844&0.698204&0.527104\\
		\hline
		
		&100 &2.603723&3.956911&0.369535&3.266323&0.662601\\
		&300 &1.197069&1.865351&0.177425&1.859670&0.662601\\
		$\beta$=1.5&500 &0.834616&1.327955&0.128104&1.497217&0.662601\\
		&700 &0.640028&1.045640&0.102552&1.302629&0.662601\\
		&1000&0.436865&0.751801&0.075883&1.099466&0.662601\\
		\bottomrule
	\end{tabular}
\end{table}
\section{Applications}
In this section, we present applications of the studied concept (WFGCRI) to chaotic dynamical maps and to the assessment of market uncertainty by employing it under the PHR model. We further demonstrate the superiority of the  WFGCRI over the traditional WCRI in capturing structural complexity and uncertainty.
\subsection{Application of WFGCRI under PHR model}
Here, we present two applications of the WFGCRI under the PHR model. First, we demonstrate that WFGCRI can effectively detect chaotic behavior in deterministic dynamical systems. Second, we apply this measure to financial data to quantify uncertainty in stock price movements.
\subsubsection{Application to Chaotic Maps}
We begin by empirically estimating the WFGCRI under the PHR model for chaotic maps. Chaos theory is instrumental in describing complex systems that exhibit high sensitivity to initial conditions and seemingly random behavior. In such cases, traditional statistical or computational models may fall short. The WFGCRI under PHR, by incorporating both distributional and temporal characteristics, serves as a robust measure for identifying the underlying uncertainty in chaotic time series. In this study, we consider two chaotic maps: the Ricker  and Tent maps.
\subsubsection*{Ricker Map}	
The Ricker map \citep{Rick1954} is defined as $
x_{n+1} = x_n \, \mathrm{e}^{r(1 - x_n)},
$
where $r > 0$ is a growth rate parameter. This map models population dynamics with density-dependent regulation and has two fixed points: $x = 0$ and $x = 1$. The fixed point at $x = 0$ is always unstable for $r > 0$, while the fixed point at $x = 1$ is locally stable when $0 < r < 2$. As $r$ increases beyond 2, the system undergoes a period-doubling cascade, eventually leading to chaotic behavior.
To generate simulated data with a sample size of $n = 10{,}000$, we use an initial value $x_0 = 0.01$ and let $r \in [1, 5]$. The bifurcation diagram (see Figure \hyperref[fig4]{5(a)}) of the Ricker map illustrates the transition from stable fixed points to periodic cycles and ultimately to chaos, depending on the value of $r$. For example, the system converges to a stable fixed point at $r = 1$, while it exhibits chaotic behavior at $r = 3$, $r = 4$, and $r = 5$.
The graphical representation of the empirical WFGCRI measure defined in equation~(\ref{eq4.2}) is presented in Figure~\hyperref[fig4]{6(a)} for $\alpha = 0.5$, with $\beta \in (0.01,5]$ and selected values of the Ricker map parameter $r = 1, 3.1, 3.5, 4.0, 4.5,$ and $4.9$.
From Figure~\hyperref[fig4]{6(a)}, it is observed that as the chosen values of $r$ increase, the WFGCRI measure also increases, with the highest curve corresponding to $r=4.9$. This behavior indicates that the uncertainty quantified by WFGCRI becomes larger for higher values of the Ricker map parameter, reflecting increased irregularity and complexity in the system dynamics.
			\begin{figure}[] \label{fig4}
	\centering
	\begin{minipage}[b]{0.45\linewidth}
		\includegraphics[height=7cm,width=8.3cm]{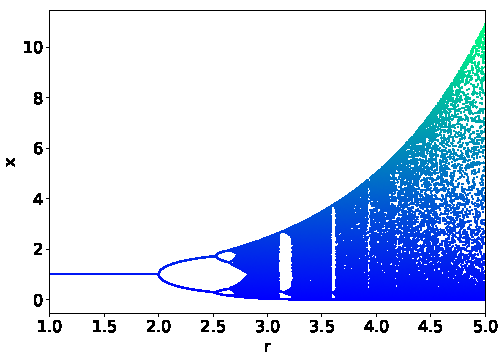}
		\centering{(a)  }
	\end{minipage}
	\quad
	\begin{minipage}[b]{0.45\linewidth}
		\includegraphics[height=7cm,width=8.3cm]{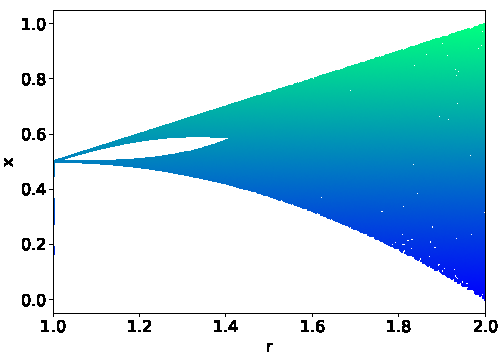}
		\centering{(b)}
	\end{minipage}
	\caption{Bifurcation diagrams of Ricker (left) and Tent maps (right).}
\end{figure}
			\begin{figure}[] 
	\centering
	\begin{minipage}[b]{0.47\linewidth}
		\includegraphics[height=7cm,width=8.5cm]{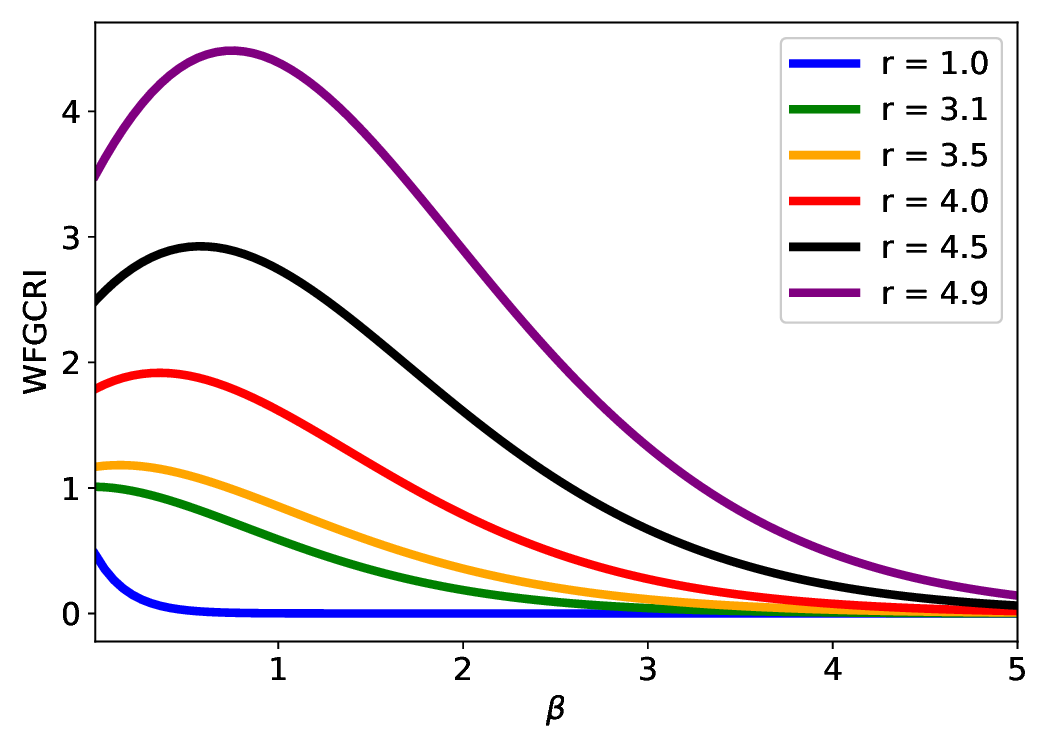}
		\centering{(a)  }
		
	\end{minipage}
	\quad
	\begin{minipage}[b]{0.47\linewidth}
		\includegraphics[height=7cm,width=8.5cm]{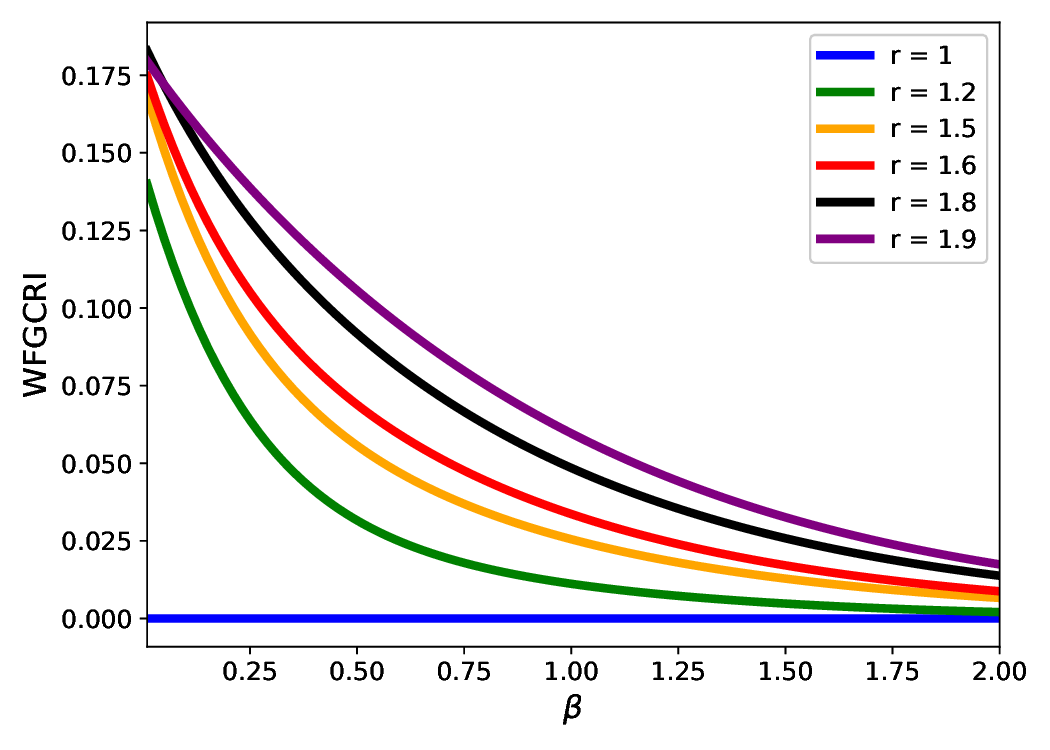}
		\centering{(b)}
	\end{minipage}
\caption{Plots of $\widehat{\mathcal K}_\beta\left({S}_n, {S}_n^\alpha;\psi\right)$ for the Ricker (left) and Tent maps (right), for selected values of $r$.}
\label{fig5}
\end{figure}
\subsubsection*{Tent Map}
The Tent map is a simple piecewise linear dynamical system defined by $$x_{n+1} =
\begin{cases}
	r x_n, & 0 \le x_n < \tfrac{1}{2}, \\
	r (1 - x_n), & \tfrac{1}{2} \le x_n \le 1,
\end{cases}$$
where \( r \in [0,2] \) is the control parameter. Despite its simplicity, the Tent map exhibits rich dynamical behavior, ranging from stable dynamics to fully developed chaos as the parameter \( r \) increases. The chaotic behavior arises from the stretching and folding mechanism inherent in the map.
To generate simulated data of size \( n = 10{,}000 \), we fix the initial condition \( x_0 = 0.01 \) and vary the control parameter \( r \) over the interval \( [1,2] \). The bifurcation diagram in Figure \hyperref[fig4]{5(b)} depicts the progression from regular dynamics to chaotic behavior as the control parameter $r$ approaches its upper limit.
We compute the WFGCRI measure for the Tent map to quantify the underlying dynamical uncertainty. In particular, the WFGCRI is plotted in Figure \hyperref[fig5]{6(b)} for selected values of $r = 1.0, 1.2, 1.5, 1.6, 1.8$ and $1.9$, with $\beta \in (0.01,2]$ and $\alpha = 0.5$.
The results show that the WFGCRI is nearly zero for lower values of $r$ (e.g., $r=1$), indicating minimal dynamical uncertainty. As the selected values of $r$ increase, the WFGCRI curves shift upward across $\beta,$ and the measure approaches $2$, reflecting increased uncertainty and the emergence of chaotic behavior in the system.
\subsubsection{Detecting Market Uncertainty using WFGCRI}
In this subsection, we apply the WFGCRI under the PHR model for $\beta \in (0.01,2]$ with $\alpha=5$ and 10 to assess temporal uncertainty in the Indian stock market, using daily closing prices of the Nifty 50 index from January 1, 2000, to December 31, 2024. This dataset comprises approximately 5,200 daily observations. Let \( P_t \) denote the closing price on day \( t \); the daily log returns are computed as 
\begin{align}\label{5.1}
    R_t = \ln(P_t) - \ln(P_{t-1}),
\end{align}
which are then shifted to non-negative values by defining \( Z_t = R_t - \min_t R_t \), thereby satisfying the non-negativity requirement of WFGCRI.
To capture the evolving structure of uncertainty, we compute WFGCRI over a rolling window of 250 trading days, with updates every 100 days. This yields a time series of WFGCRI estimates that reflect the market’s information uncertainty across different time horizons and fractional orders.
Figure~\ref{fig7} presents the time series of \( R_t \), where significant negative spikes mark episodes of acute market stress. Noteworthy among these are: 
\begin{itemize}
	\item \textbf{May 2004 (UPA 1 election crash)}: The BSE plunged by 15.52\% in a single day, marking its largest percentage fall ever.
	\item \textbf{2006–2009 financial turmoil}: Including January 2008’s “Black Monday”, March 2008 declines, and the major crash of October 2008.
	\item \textbf{March 2020 (COVID-19 pandemic)}: A sequence of historic losses culminating in a 3,900-point drop on March 23.
\end{itemize}
For further historical context, refer \url{https://en.wikipedia.org/wiki/Stock_market_crashes_in_India}.
Figure~\ref{fig8} illustrates the contour plot of the WFGCRI, where the vertical axis represents the fractional order $\beta \in (0.01,2]$  and the horizontal axis denotes time. The color intensity reflects the magnitude of the WFGCRI, which captures the  uncertainty associated with the shifted return process $Z_t$. It is evident from the figure that the WFGCRI is capable of effectively identifying structural breaks and crisis periods in financial time series, as highlighted by the pronounced contrasts and fine resolution observed in the contour patterns. Therefore, the WFGCRI serves as a simple yet powerful tool for tracking temporal variations in uncertainty and detecting market instability.
\begin{figure}[H]
	\centering
	\includegraphics[width=5.3in, height=3in]{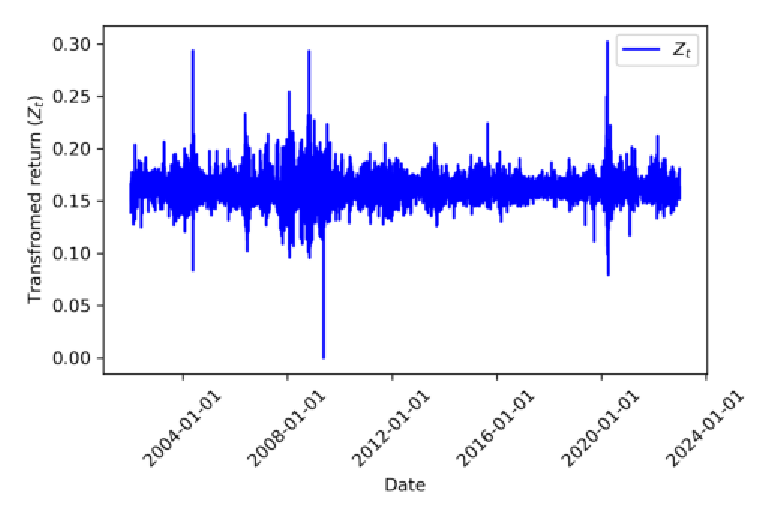}
	\caption{Transformed price returns of the Nifty 50 dataset.
	}\label{fig7}
\end{figure}
			\begin{figure}[H] 
	\centering
	\begin{minipage}[b]{0.47\linewidth}
		\includegraphics[height=7cm,width=8.6cm]{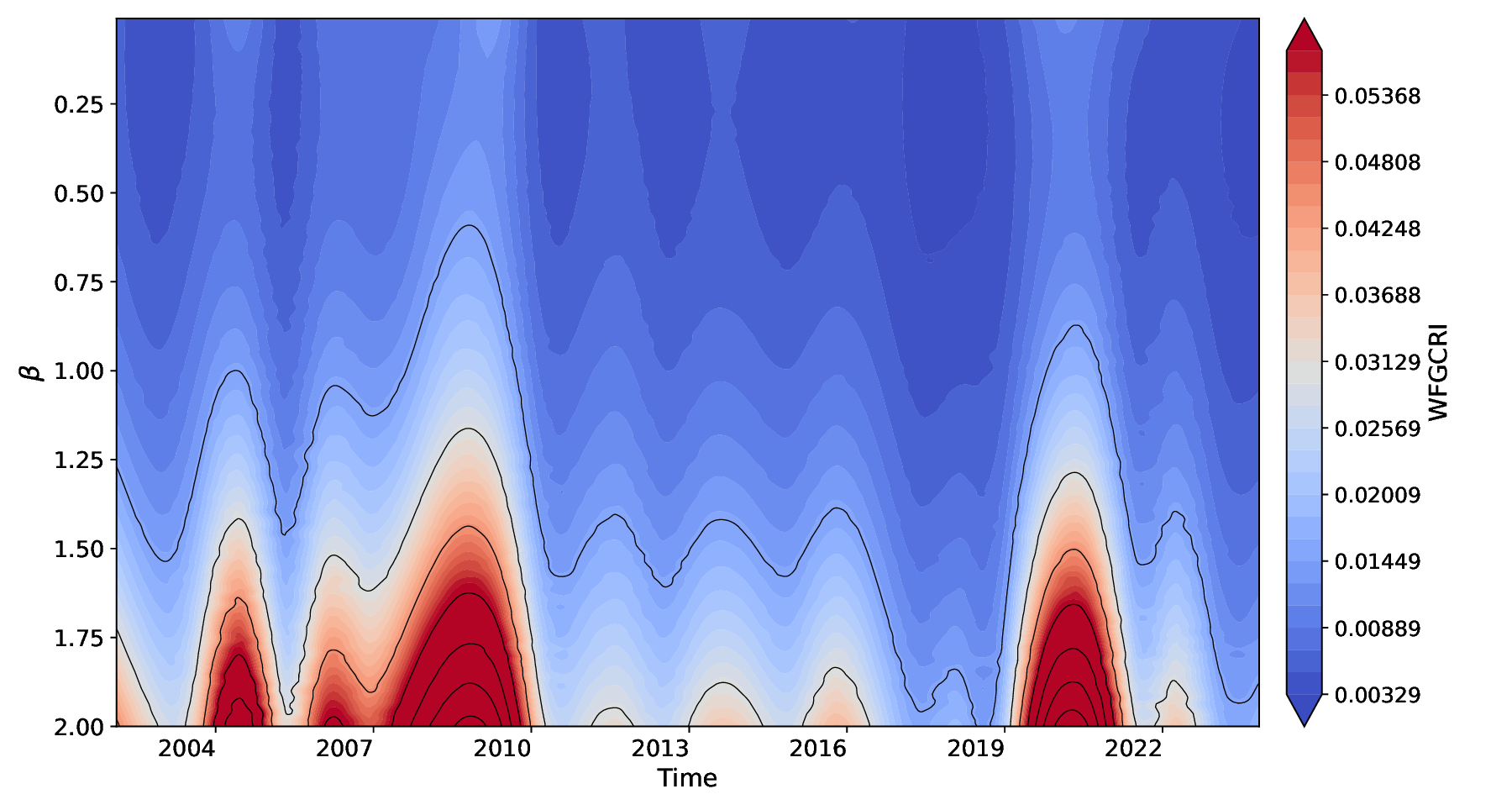}
		\centering{(a)}
	\end{minipage}
	\quad
	\begin{minipage}[b]{0.47\linewidth}
		\includegraphics[height=7cm,width=8.6cm]{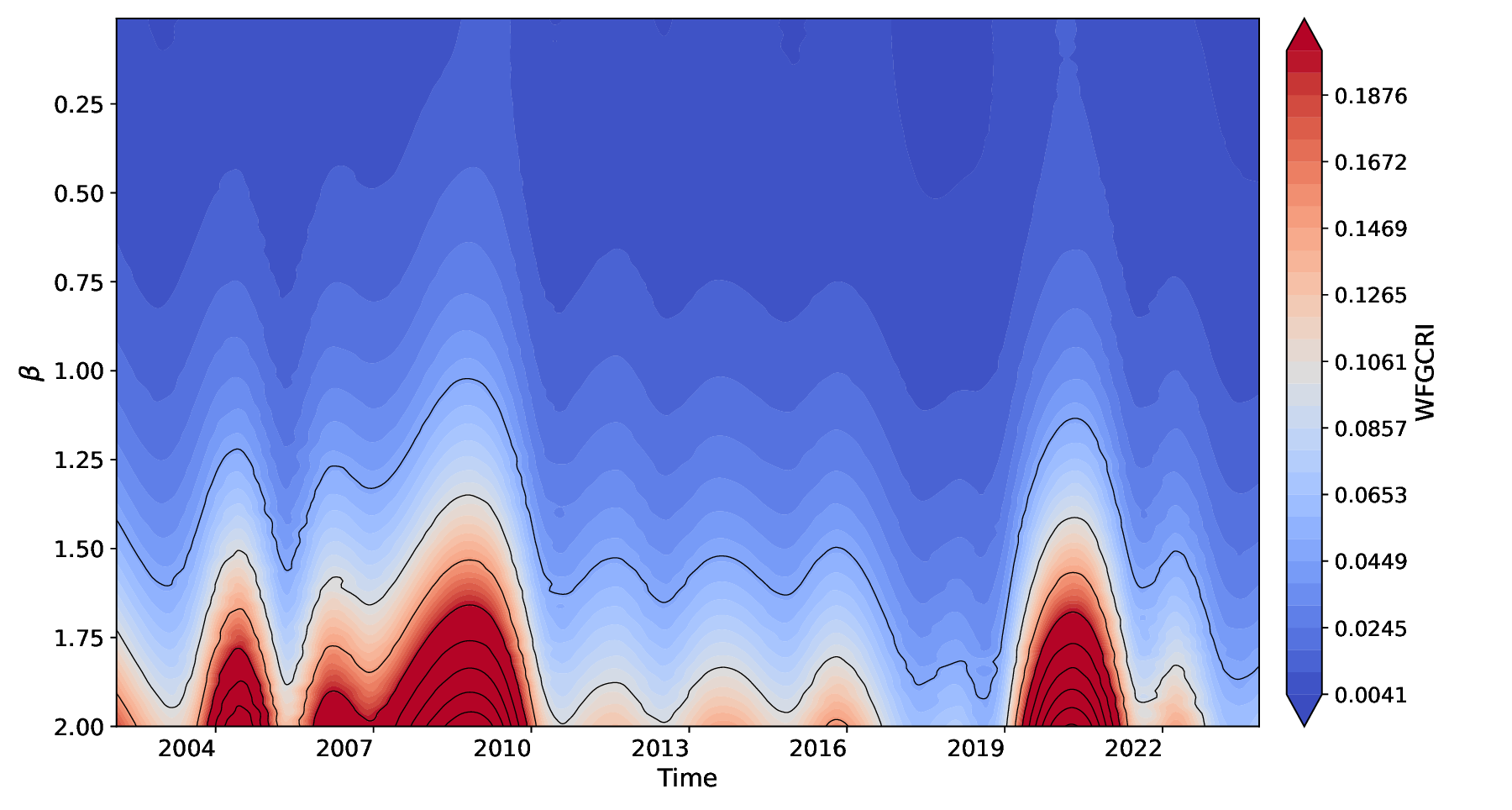}
		\centering{(b)}
	\end{minipage}
	\caption{Contour plot of WFGCRI under PHR model with $\alpha=5$ (left) and $\alpha=10$ (right) for $\beta \in (0.01,2]$ (with a step size of 0.01) and time. 
}\label{fig8}
\end{figure}
\subsection{Performance Comparison of WFGCRI and WCRI under Stochastic Market Dynamics}
In this subsection, we present a comparative analysis of the WFGCRI and WCRI measures under stochastic market conditions. We utilize the Nifty Auto and Nifty Energy indices from the Indian stock market over the period 1990–2022, obtained from Kaggle. Both datasets are transformed according to (\ref{5.1}), and the resulting series are illustrated in Figure \ref{fig9}. Here, we consider the Nifty Auto index as the actual data and the Nifty Energy index as the reference data.
The WFGCRI is computed using the estimation procedure proposed in~(\ref{4.3}) and is analyzed over the range $\beta \in (0,5]$, as illustrated in Figure~\ref{fig10}.
The results indicate that for $\beta \in (0.01,5]$, the WFGCRI remains close to zero and provides slightly larger values when $0 < \beta < 1$. Since WFGCRI coincides with WCRI at $\beta = 1$, it attains a comparatively low value in this region and is therefore less sensitive to identifying structural discrepancies or regime variations in the market data.
 In contrast, the WFGCRI, owing to its fractional weighting scheme, maintains sensitivity over a broad range of $\beta$ values and effectively captures deviations induced by market irregularities.
Thus, WFGCRI offers a more informative and discriminative inaccuracy measure compared to WCRI, especially in detecting subtle discrepancies that the classical measure overlooks. Furthermore, Figure \ref{fig10} also highlights the asymmetric behavior of the WFGCRI with respect to $\beta$.
\begin{figure}[H]
    \centering
    \includegraphics[width=5in, height=2.4in]{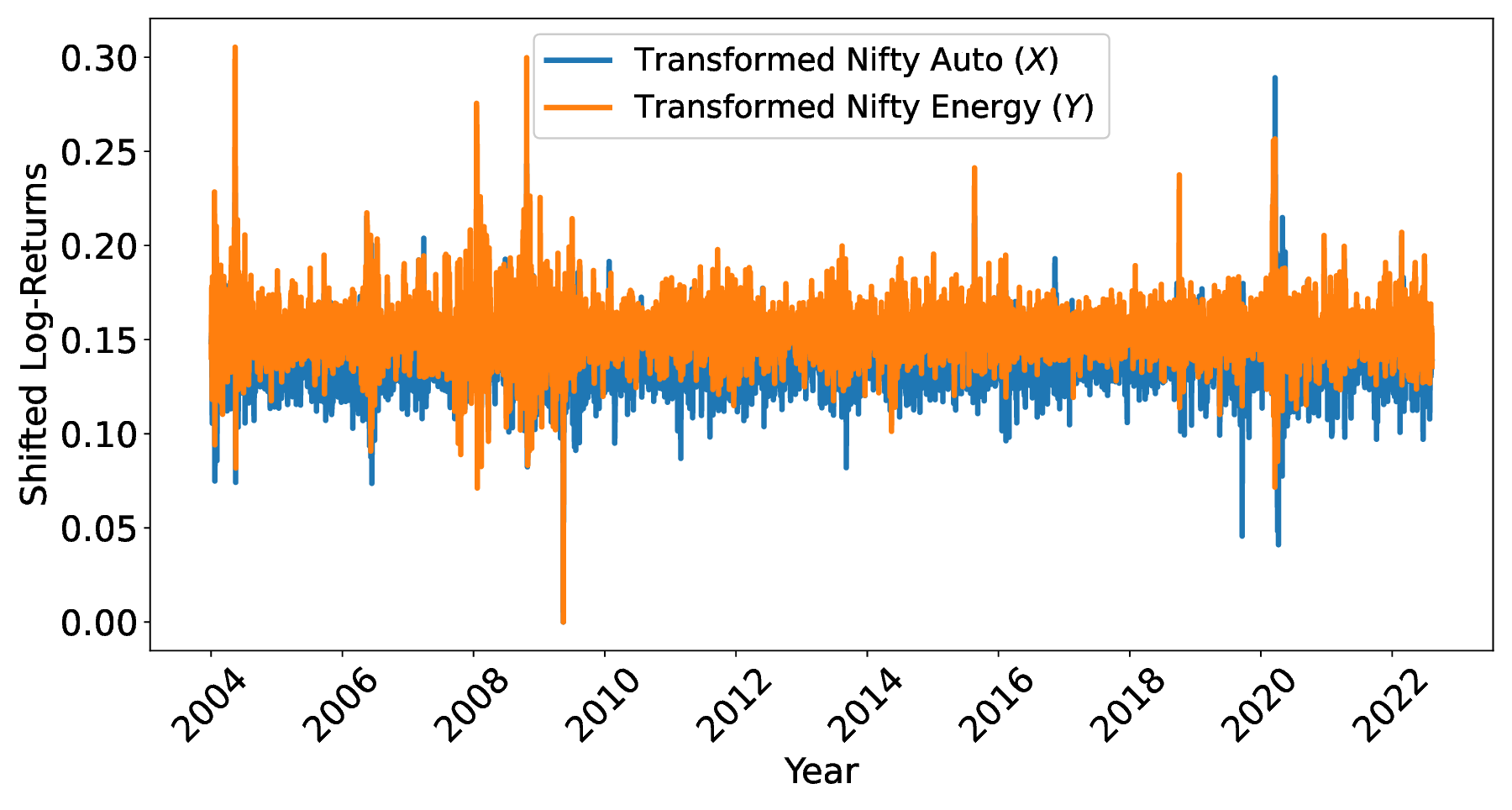}
    \caption{Transformed returns of Nifty Auto and Nifty Energy datasets.}\label{fig9}
\end{figure}
\begin{figure}[H]
	\centering
	\includegraphics[width=5in, height=2.4in]{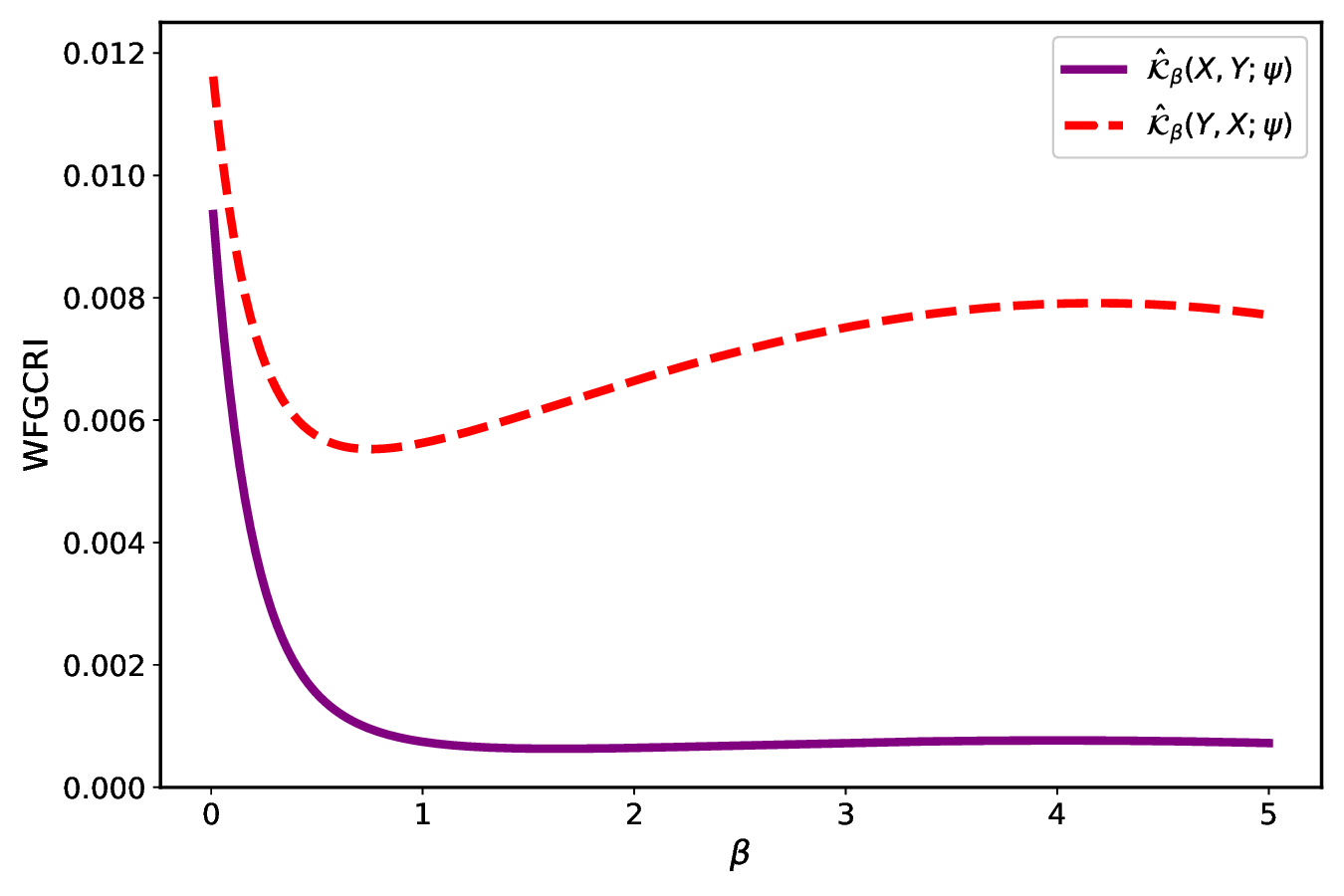}
	\caption{Plot of WFGCRI between Nifty Energy and Nifty Auto datasets.}\label{fig10}
\end{figure}
\section{Conclusion}
Motivated by recent advances in weighted inaccuracy measures, this paper developed and investigated a new weighted inaccuracy measure, referred to as WFGCRI. The behavior of the WFGCRI is studied under various stochastic orders, and several bounds are derived. Its relationship with WFGCRE is explored, providing deeper insight into its mathematical structure. Moreover, the effects of monotonic and affine transformations on the WFGCRI are examined, demonstrating its robustness under commonly encountered data transformations. The behavior of the WFGCRI is also analyzed under mixture hazard rate models.
To address time-varying uncertainty, a dynamic version of the WFGCRI was introduced and studied under PHR and PO models with illustrative examples.
A non-parametric estimation procedure for the WFGCRI was developed and assessed through extensive simulation studies. The results demonstrated that AB, RMSE, and the length of 95\% CI decreased as the sample size increased, thereby confirming the consistency and practical feasibility of the proposed estimator.
The usefulness of the WFGCRI was further illustrated through applications to chaotic dynamical systems, specifically the Ricker  and  Tent maps. In both cases, the WFGCRI effectively captured the underlying chaotic dynamics, underscoring its potential as a diagnostic tool for complex nonlinear systems.
Finally, the practical relevance of the WFGCRI under the PHR model is demonstrated using real financial time-series data, specifically the daily share prices of the Nifty 50 index. The empirical findings reveal that the WFGCRI under PHR attains higher values during periods of market instability and financial crises, thereby effectively capturing structural changes in the data. In addition, an empirical estimation procedure for the WFGCRI is proposed and evaluated through Monte Carlo simulation. Furthermore, applications to the Nifty Auto (true data) and Nifty Energy (reference data) indices show that the WFGCRI consistently outperforms the WCRI in measuring the discrepancy between the two series, demonstrating its superior sensitivity and effectiveness in real-world applications.

As future research directions, the proposed WFGCRI can be extended to multivariate and high-dimensional settings. Additionally, its applicability may be explored in other domains such as reliability engineering, survival analysis, and anomaly detection in complex systems.
\section*{Funding Sources}
This research did not receive any specific grant from funding agencies in the public, commercial, or not-for-profit sectors.
\section*{Declaration of AI}
During the preparation of this work, the authors used ChatGPT (OpenAI) to improve the clarity and language of the manuscript. The authors reviewed and edited the content as needed and take full responsibility for the content of the published article.
\section*{Conflict of Interest}
The authors declare that they have no conflict of interest concerning the publication of this manuscript.
\section*{Data Availability}
\footnotesize\begin{itemize}
	\item The Nifty 50 dataset:\\ \url{https://www.kaggle.com/datasets/adarshde/nifty-50-data-1990-2024}.
	
	\item	The Nifty Energy and Nifty Auto data sets:\\ \url{https://www.kaggle.com/datasets/debashis74017/stock-market-index-data-india-1990-2022}.
\end{itemize}

\end{document}